\documentclass[review,12pt]{elsarticle}
\usepackage{amsmath}
\usepackage{amsfonts}
\usepackage{geometry}   
\usepackage[parfill]{parskip}    
\usepackage{graphicx}
\usepackage{amssymb}
\usepackage{epstopdf}
\usepackage{psfrag}
\newcommand{\figref}[1]{{Figure~\ref{#1}}}
\newcommand {\bea}{\begin{eqnarray}}
\newcommand {\ea}{\end{eqnarray}}

\newtheorem{theorem}{Theorem}[section]
\newtheorem{Assumption}{Assumption}[section]
\newtheorem{lemma}{Lemma}[section]
\newtheorem{remark}{Remark}[section]

\newtheorem{corollary}{Corollary}[section]

\newenvironment{proof}[1][Proof]{\textbf{#1.} }{\hspace{\stretch{1}}\rule{0.5em}{0.5em}}

\usepackage{xspace}

\usepackage{subfigure}

\newcommand{\thmref}[1]{{Theorem~\ref{#1}}}
\newcommand{\lemref}[1]{{Lemma~\ref{#1}}}
\newcommand{\secref}[1]{{Section~\ref{#1}}}
\newcommand{\assref}[1]{{Assumption~\ref{#1}}}

\newcommand{\coref}[1]{{Corollary~\ref{#1}}}
\newcommand{\rmref}[1]{{Remark~\ref{#1}}}

\journal{Applied Numerical Mathematics}
\begin{document}
\begin{frontmatter}
\title{Strong convergence of the linear implicit Euler method for the finite element discretization of  semilinear   non-autonomous SPDEs driven by 
multiplicative or additive  noise}

\author[jdm]{Jean Daniel Mukam}
\ead{jean.d.mukam@aims-senegal.org}
\address[jdm]{Fakult\"{a}t f\"{u}r Mathematik, Technische Universit\"{a}t Chemnitz, 09107 Chemnitz, Germany}

\author[at,atb,atc]{Antoine Tambue}
\cortext[cor1]{Corresponding author}
\ead{antonio@aims.ac.za}
\address[at]{Department of Computing Mathematics and Physics,  Western Norway University of Applied Sciences, Inndalsveien 28, 5063 Bergen, Norway.}
\address[atb]{Center for Research in Computational and Applied Mechanics (CERECAM), and Department of Mathematics and Applied Mathematics, University of Cape Town, 7701 Rondebosch, South Africa.}
\address[atc]{The African Institute for Mathematical Sciences(AIMS) of South Africa,
6-8 Melrose Road, Muizenberg 7945, South Africa.}

%

\begin{abstract}
This paper aims to investigate the numerical approximation of semilinear non-autonomous stochastic partial differential equations (SPDEs) 
driven by multiplicative or additive noise. Such equations are more realistic than the autonomous ones when modeling real world phenomena. Such equations  find applications  in many  fields such as  transport  in porous media, quantum fields theory, electromagnetism and 
nuclear physics.
Numerical approximations of autonomous SPDEs are thoroughly investigated in the literature, while the   non-autonomous 
case is not yet well understood.
Here,  a  non-autonomous  SPDE  is discretized in space by the finite element method and in time by the linear implicit Euler method.
We break  the  complexity in the analysis of the time  depending,  not   necessarily self-adjoint linear operators with the corresponding 
semigroup  and  provide the strong convergence 
result of the fully discrete scheme toward the mild solution.
The results indicate how the converge order  depends on the regularity of the initial solution and the noise. Additionally, for additive noise we achieve convergence order in time approximately $1$ under less regularity assumptions on the nonlinear drift term than required in the current literature, even in the autonomous case. Numerical simulations motivated  from  realistic  porous  media flow are provided to illustrate our theoretical finding.
\end{abstract}

\begin{keyword}
Linear implicit Euler method\sep Stochastic partial differential equations\sep Multiplicative \& Additive  noise\sep Strong  convergence\sep Finite  element method\sep Non-autonomous problems.

\end{keyword}
\end{frontmatter}

\section{Introduction}
\label{intro}
We consider numerical approximation of the  following non-autonomous SPDE defined in $\Lambda\subset \mathbb{R}^d$, $d=\{1,2,3\}$ (where $\Lambda$ is bounded with smooth boundary), 
\begin{eqnarray}
\label{model}
dX(t)+A(t)X(t)dt=F(t,X(t))dt+B(t,X(t))dW(t), \quad t\in(0,T],\quad X(0)=X_0, 
\end{eqnarray}
 on the Hilbert space $L^2(\Lambda, \mathbb{R})$,  $T>0$  is the final time, $F$ and $B$ are nonlinear functions and $X_0$ is the initial data, which  is random.
The  family of linear operators $A(t)$ are unbounded, not necessarily self-adjoint, and for all $s\in[0,T]$,  $-A(s)$ is  the generator of an analytic semigroup
$S_s(t)=:e^{-t A(s)}, t\geq 0.$ The noise  $W(t)$ is a $Q-$Wiener process defined in a filtered probability space $(\Omega,\mathcal{F}, \mathbb{P}, \{\mathcal{F}_t\}_{t\geq 0})$.
The filtration is assumed to fulfill the usual conditions (see e.g. \cite[Definition 2.1.11]{Prevot}).
Note  that the noise can be represented as follows (see e.g. \cite{Prevot,Prato})
\begin{eqnarray}
\label{noise}
W(x,t)=\sum_{i\in \mathbb{N}^d}\sqrt{q_i}e_i(x)\beta_i(t), \quad t\in[0,T],\quad x\in \Lambda,
\end{eqnarray} 
where $q_i$, $e_i$, $i\in\mathbb{N}^d$ are respectively the eigenvalues and the eigenfunctions of the covariance operator $Q$,
and $\beta_i$, $i\in\mathbb{N}$, are independent and identically distributed standard Brownian motions. Precise assumptions on $F$, $B$, $X_0$ and $A(t)$ 
will be given in the next section to ensure the existence of the unique mild solution $X$ of \eqref{model}. In many situations it is 
hard to exhibit explicit  solutions of  SPDEs. Therefore, numerical algorithms are good tools to provide realistic approximations.
Strong  approximations of autonomous SPDEs with constant  linear self-adjoint operator $A(t)=A$ are widely 
investigated in the literature, see e.g.  \cite{Xiaojie1,Xiaojie2,Kruse1,Arnulf2,Antonio3,Yan1} and references therein.
When we turn our attention to the case of semilinear SPDEs, still with constant operator $A(t)=A$, but not necessary self-adjoint,
the list of references becomes remarkably short, see e.g. \cite{Antonio1,Antjd1}.
Note that modelling real world phenomena with time dependent linear operator is more realistic than modelling with
time independent linear operator (see e.g. \cite{Blanes} and references therein).  The deterministic counterpart of \eqref{model} 
finds  applications in many fields such as quantum fields theory, electromagnetism, 
nuclear physics and  transport in porous media.
To the best of our knowledge, numerical approximations of non-autonomous SPDEs
are not yet well understood in the literature due to the complexity of the linear operator $A(t)$, its semigroup $S_t(s)$ and the resolvent operator $(\mathbf{I}+tA(s))^{-1}$, $t, s\in[0, T]$.
Our aims is to fill that gap in this paper and in our accompanied papers \cite{Antjd2,Antjd4}.
The Magnus-type integrators are developed in the accompanied papers \cite{Antjd2,Antjd4} for SPDEs with multiplicative  and additive noise.
  Magnus-type integrators use the fact that the solution of the deterministic system of  differential equations $y'(t)=A(t)y(t)$ can be represented in the following 
exponential form $y(t)=\exp(\Gamma(t))y(0)$ (see e.g. \cite{Magnus,Blanes2,Blanes1}), where $\Gamma(t)$  is called Magnus expansion or Magnus series. Note that the Magnus  series does not always converges. One sufficient condition for its convergence is:  $\int_0^T\Vert A(t)\Vert_2 dt<\pi$ (see e.g. \cite{ConvMagnus, Magnus} or \cite[Section IV.7]{Hairer}),  where $\Vert .\Vert_2$ stands for the matrix norm. 
For some problems with large $\Vert A(t)\Vert_2$, the Magnus series seems to  diverge (see e.g. \cite{Ostermann1}).
Hence, for such problems, it is important to find alternative numerical schemes.
In this paper, we develop an alternative method based on linear implicit method, which does not make use of the Magnus series and which is more stable than the explicit Magnus-type integrators developed in \cite{Antjd2,Antjd4}. The space discretization is performed using the finite element method.
Note that the implementation of this method is based on the resolution of linear systems and  may be more efficient than Magnus-type integrators 
when the appropriate preconditioners are used. Here, we break the  complexity  in the analysis of the time  depending,  not   necessarily self-adjoint linear operators 
with the corresponding semigroup  and  provide the strong convergence results of the fully discrete schemes toward the exact solution in  the root-mean-square $L^2$ norm.
%
%
%
%
The main  challenge here is that the resolvent operators change at each time step. So novel stability estimates, useful in the convergence analysis are needed. 
These novel estimates are provided in \secref{Preparation}. Note that the preparatory results in \secref{Preparation} are different from  results in \cite{Antjd2,Antjd4} and are very challenging.
\begin{itemize}
\item In fact, here the key ingredient is the discrepancy between the  two parameters semigroup $U_h(t,s)$ and the  resolvent operator $\left(\mathbf{I}+\Delta tA_h(t_j)\right)^{-1}$, which is much more complicated than estimating the discrepancy between $U_h(t,s)$ and its approximated form $e^{A_h(s)(t-s)}$, which was one of the  core of  works  in \cite{Antjd2,Antjd4}.
\item Note that even in the autonomous  case, the discrepancy between the semigroup $S(t)$ and the resolvent operator $(\mathbf{I}+\Delta tA)^{-1}$ is a key ingredient in approximating SPDEs with linear implicit method. Such discrepancy for smooth and non smooth initial data  with linear constant self adjoint operator $A$ were done in \cite[Theorem 7.8]{Thomee} and \cite[Theorem 7.7]{Thomee} respectively, where  authors used the spectral decomposition of $A$. In fact, \cite[Theorem 7.8]{Thomee} and \cite[Theorem 7.7]{Thomee} are key ingredients in the literature when analyzing convergence of autonomous SPDEs via linear implicit method, see e.g. \cite{Xiaojie2,Kruse1,Antonio3}.
\item In the case of non-autonoumous SPDEs, \cite[Theorem 7.8]{Thomee} and \cite[Theorem 7.7]{Thomee} are no longer applicable and to prove their analogous for time dependent operator,  one cannot  just follow the steps of the proofs  of \cite[Theorem 7.8]{Thomee} and \cite[Theorem 7.7]{Thomee}, since in this case, in addition to the fact that the operators $A(t)$ are changing at each time step, they are not   self adjoint and therefore the spectral decomposition is not applicable.  \secref{Preparation} (more precisely Lemmas \ref{lemma7} and  \ref{lemma8}) uses an argument based on  telescopic sums and provides key ingredients to handle these challenges. 
\item Moreover, in comparison to many works in the literature for additive noise, where the authors achieved convergence order in time approximately $1$ (see e.g. \cite{Xiaojie2,Antonio3}), we also achieve similar convergence order, but with less regularity assumptions on the nonlinear drift function, which extends the class of the nemytskii operaor $F$.  In fact, we only require $F$ to be differentiable with Lipschitz continuous derivative, while in the up to date literature (see e.g. \cite{Xiaojie2,Antonio3,Xiaojie1}) the authors requires the derivatives up to order $2$ to be bounded. This is restrictive and exclude many Nemytskii operators such as $F(u)=\frac{\vert u\vert}{1+u^2}$,  $u\in H$. In fact, the later function $F$  does not even have a second derivative at $0$.
\end{itemize} 
Our rigorous mathematical analysis shows how the convergence rates depend on the regularity of the initial data and the noise. 
In fact, we achieve  convergence orders $\mathcal{O}\left(h^{\beta}+\Delta t^{\frac{\min(\beta,1)}{2}}\right)$ 
for multiplicative noise and $\mathcal{O}\left(h^{\beta}+\Delta t^{\frac{\beta}{2}-\epsilon}\right)$ for additive noise, where $\beta$ is  
the regularity parameter from \assref{assumption1}  and $\epsilon$ is a positive number small enough. 

The rest of this paper is organized as follows:  Section \ref{wellposed} deals with the well posedness problem, the numerical 
scheme and the main results. In section \ref{convergenceproof}, we provide some errors estimates for 
the corresponding deterministic homogeneous problem as preparatory results along with the proofs  of the main results.
Section \ref{numexperiment} provides some numerical experiments  motivated  from  realistic  porous  media to sustain the theoretical findings.

\section{Mathematical setting and main results}
\label{wellposed}
\subsection{Main assumptions and well posedness problem}
\label{notation}
Let  $(H,\langle.,.\rangle,\Vert .\Vert)$ be a separable Hilbert space.  For any $p\geq 2$ and for a Banach space $U$,
we denote by $L^p(\Omega, U)$ the Banach space of all equivalence classes of $p$ integrable $U$-valued random variables. Let $L(U,H)$ be 
 the space of bounded linear mappings from $U$ to $H$ endowed with the usual  operator norm $\Vert .\Vert_{L(U,H)}$. By  $\mathcal{L}_2(U,H):=HS(U,H)$,
 we  denote the space of Hilbert-Schmidt operators from $U$ to $H$ equipped with the norm 
 $\Vert l\Vert^2_{\mathcal{L}_2(U,H)}:=\sum\limits_{i=1}^{\infty}\Vert l\psi_i\Vert^2$,  $l\in \mathcal{L}_2(U,H)$, where $(\psi_i)_{i=1}^{\infty}$ is an orthonormal basis of $U$. Note that this definition is independent of the orthonormal basis of $U$.  For simplicity, we use the notations $L(U,U)=:L(U)$ and $\mathcal{L}_2(U,U)=:\mathcal{L}_2(U)$. 
 For all $l\in L(U,H)$ and $l_1\in\mathcal{L}_2(U)$, it holds that
\begin{eqnarray}
\label{chow1}
ll_1\in\mathcal{L}_2(U,H)\quad \text{and} \quad  \Vert ll_1\Vert_{\mathcal{L}_2(U,H)}\leq \Vert l\Vert_{L(U,H)}\Vert l_1\Vert_{\mathcal{L}_2(U)},
\end{eqnarray}
see e.g. \cite{Chow}. 
   The covariance operator  $Q: H\longrightarrow H$ is assumed to be positive definite and self-adjoint. 
 The space of Hilbert-Schmidt operators from  $Q^{\frac{1}{2}}(H)$ to $H$ is denoted by $L^0_2:=\mathcal{L}_2(Q^{\frac{1}{2}}(H),H)=HS(Q^{\frac{1}{2}}(H),H)$. As usual, $L^0_2$ is equipped with the norm $\Vert l\Vert_{L^0_2} :=\Vert lQ^{\frac{1}{2}}\Vert_{HS}=\left(\sum\limits_{i=1}^{\infty}\Vert lQ^{\frac{1}{2}}e_i\Vert^2\right)^{\frac{1}{2}}, \quad  l\in L^0_2$,
where $(e_i)_{i=1}^{\infty}$ is an orthonormal basis  of $H$.
This definition is independent of the orthonormal basis of $H$. 

For an $L^0_2$- predictable stochastic process $\phi :[0,T]\times \Lambda\longrightarrow L^0_2$ such that
\begin{eqnarray*}
\int_0^t\mathbb{E}\left\Vert \phi(s) Q^{\frac{1}{2}}\right\Vert^2_{\mathcal{L}_2(H)}ds<\infty,\quad t\in[0,T],
\end{eqnarray*}
the following relation called It\^{o}'s isometry holds
\begin{eqnarray}
\label{ito}
\mathbb{E}\left\Vert\int_0^t\phi(s) dW(s)\right\Vert^2=\int_0^t\mathbb{E}\Vert \phi(s)\Vert^2_{L^0_2}ds=\int_0^t\mathbb{E}\left\Vert\phi(s) Q^{\frac{1}{2}}\right\Vert^2_{\mathcal{L}_2(H)}ds,\quad t\in[0,T],
\end{eqnarray}
see e.g. \cite[Step 2 in Section 2.3.2]{Prato} or \cite[Proposition 2.3.5]{Prevot}.

In the rest of this paper, we consider $H=L^2(\Lambda,\mathbb{R})$.
To guarantee the existence of a unique mild solution of \eqref{model} and for the purpose of the  convergence analysis, we make the following  assumptions.
\begin{Assumption}
 \label{assumption1}
The initial data $X_0 : \Omega\longrightarrow H$ is assumed to be measurable and belongs to $L^2\left(\Omega , \mathcal{D}\left(\left(A(0)\right)^{\frac{\beta}{2}}\right)\right)$, with  $0\leq \beta\leq 2.$ 
 \end{Assumption}
\begin{Assumption}
\label{assumption2}
\begin{itemize}
\item[(i)]
As in \cite{Ostermann1,Ostermann2,Antjd2,Praha}, we assume that $\mathcal{D}(A(t))=D$, $0\leq t\leq T$ and 
that the family of linear operators $A(t) : D\subset H\longrightarrow H$ is uniformly sectorial on $0\leq t\leq T$, i.e. 
there exist constants $c>0$ and $\theta\in(\frac{1}{2}\pi,\pi)$ such that
\begin{eqnarray*}
\Vert (\lambda\mathbf{I}-A(t))^{-1}\Vert_{L(L^2(\Lambda))}\leq \frac{c}{\vert \lambda\vert},\quad \lambda\in S_{\theta},
\end{eqnarray*}
where $S_{\theta}:=\{\lambda\in\mathbb{C} : \lambda=\rho e^{i\phi}, \rho>0, 0\leq \vert \phi\vert\leq \theta\}$. 
As in \cite{Ostermann2}, by a standard scaling argument, we assume $-A(t)$ to be invertible with bounded inverse.
\item[(ii)]
As in \cite{Ostermann2,Ostermann1}, we require the following Lipschitz conditions:  there exists a positive constant $K_1$ such that
\begin{eqnarray}
\label{conditionB}
\left\Vert \left(A(t)-A(s)\right)(A(0))^{-1}\right\Vert_{L(H)}&\leq& K_1\vert t-s\vert,\quad s,t\in[0, T],\\
\left\Vert  (A(0))^{-1}\left(A(t)-A(s)\right)\right\Vert_{L(D,H)}&\leq& K_1\vert t-s\vert,\quad s,t\in[0, T].
\end{eqnarray}
\item[(iii)]
 Since we are dealing with non smooth data, we follow \cite{Praha,Antjd2} and  assume that 
\begin{eqnarray}
\label{domaine}
\mathcal{D}((A(t))^{\alpha})=\mathcal{D}((A(0))^{\alpha}),\quad 0\leq t\leq T,\quad \alpha\in[0,1]
\end{eqnarray}
and there exists a positive constant $K_2$ such that  the following estimate holds 
\begin{eqnarray}
\label{equivnorme1}
K_2^{-1}\Vert (A(0))^{\alpha}u\Vert\leq \Vert (A(t))^{\alpha}u\Vert\leq K_2\Vert (A(0))^{\alpha}u\Vert, \quad u\in D\left((A(0))^{\alpha}\right),\quad t\in[0,T].
\end{eqnarray} 
\end{itemize}
\end{Assumption}
\begin{remark}
\label{remark1}
As a consequence of \assref{assumption2}, for all $\alpha\geq 0$ and $\gamma\in[0,1]$, there exists a constant $C_1>0$ such that the following estimates hold uniformly in $t\in[0,T]$
\begin{eqnarray}
\label{smooth}
\left\Vert (A(t))^{\alpha}e^{-sA(t)}\right\Vert_{L(H)}\leq C_1s^{-\alpha},s>0,\quad
\label{smootha}
 \left\Vert(A(t))^{-\gamma}\left(\mathbf{I}-e^{-rA(t)}\right)\right\Vert_{L(H)}\leq C_1r^{\gamma}, \quad r\geq 0.
\end{eqnarray}
\end{remark}
\begin{remark}
\label{remark2}
Let $\Delta(T):=\{(t,s) : 0\leq s\leq t\leq T\}$.
It is well known \cite[Theorem 6.1, Chapter 5]{Pazy} that under \assref{assumption1} there exists a unique evolution system \cite[Definition 5.3, Chapter 5]{Pazy} $U : \Delta(T)\longrightarrow L(H)$ satisfying:
\begin{itemize}
\item[(i)] there exists a positive constant $K_0$ such that
\begin{eqnarray*}
\Vert U(t,s)\Vert_{L(H)}\leq K_0,\quad 0\leq s\leq t\leq T.
\end{eqnarray*}
\item[(ii)] $U(.,s)\in C^1(]s,T] ; L(H))$, $0\leq s\leq T$,
\begin{eqnarray*}
\frac{\partial U}{\partial t}(t,s)=A(t)U(t,s), \quad 
\Vert A(t)U(t,s)\Vert_{L(H)}\leq \frac{K_0}{t-s},\quad 0\leq s<t\leq T.
\end{eqnarray*}
\item[(iii)] $U(t,.)v\in C^1([0,t[ ; H)$, $0<t\leq T$, $v\in\mathcal{D}(A(0))$ and 
\begin{eqnarray*}
\frac{\partial U}{\partial s}(t,s)v=-U(t,s)A(s)v,\quad 
 \Vert A(t)U(t,s)A(s)^{-1}\Vert_{L(H)}\leq K_0, \quad 0\leq s\leq t\leq T.
\end{eqnarray*}
\end{itemize}
\end{remark}
 We equip $V_{\alpha}(t) : = \mathcal{D}((A(t))^{\frac{\alpha}{2}})$, $\alpha\in \mathbb{R}$ with the norm $\Vert u\Vert_{\alpha,t} := \Vert (A(t))^{\frac{\alpha}{2}}u\Vert$. Due to \eqref{domaine}, \eqref{equivnorme1} and for the seek of  ease notations, we simply write $V_{\alpha}$ and $\Vert .\Vert_{\alpha}$.

We follow \cite{Praha,Antjd2} and  make the following assumptions on operators $F$ and $B$.
\begin{Assumption}
\label{assumption3}
The nonlinear operator $F : [0,T]\times H\longrightarrow H$ is  $\frac{\beta}{2}$-H\"{o}lder continuous with
respect to the first variable and Lipschitz continuous with respect to the second variable, i.e.  there exists a positive constant $K_3$ such that 
\begin{eqnarray*}
\Vert F(s,0)\Vert \leq K_3, \quad \Vert F(t, u)-F(s,v)\Vert \leq K_3\left(\vert t-s\vert^{\frac{\beta}{2}}+\Vert u-v\Vert\right), \; s,t\in[0, T],\; u, v\in H. 
\end{eqnarray*}
\end{Assumption}
 \begin{Assumption}
 \label{assumption4}
 The diffusion coefficient $B : [0,T]\times H\longrightarrow L^0_2$ is $\frac{\beta}{2}$-H\"{o}lder continuous with respect to the first variable and Lipschitz continuous with respect to the second variable, i.e. there exists a positive constant $K_4$ such that
 \begin{eqnarray*}
 \Vert B(s,0)\Vert_{L^0_2}\leq K_4,\quad \Vert B(t,u)-B(s,v)\Vert_{L^0_2}\leq K_4\left(\vert t-s\vert^{\frac{\beta}{2}}+\Vert u-v\Vert\right),\; s, t\in [0, T],\, u,v\in H.
 \end{eqnarray*}
 \end{Assumption}

 To establish  our $L^2$ strong convergence result when dealing with multiplicative noise, we will also need the following further assumption on the diffusion term when $\beta \in [1,2)$, which was also used in  \cite{Arnulf1,Kruse2} to achieve optimal  regularities, and in \cite{Antonio1,Kruse1,Antjd1,Antjd3} to achieve optimal convergence orders in space and in time.
\begin{Assumption}
\label{assumption5}
We assume that $B\left(\mathcal{D}\left(A(0)^{\frac{\beta-1}{2}}\right)\right)\subset HS\left(Q^{\frac{1}{2}}(H),\mathcal{D}\left(A(0)^{\frac{\beta-1}{2}}\right)\right)$ and there exists  $c\geq 0$ such that  for all  $v\in\mathcal{D}\left(A(0)^{\frac{\beta-1}{2}}\right)$,
$\left\Vert A(0)^{\frac{\beta-1}{2}}B(v)\right\Vert_{L^0_2}\leq c\left(1+\Vert v\Vert_{\beta-1}\right)$, where $\beta$ is the parameter defined in Assumption \ref{assumption1}.
\end{Assumption}
Typical examples which fulfill \assref{assumption5} are stochastic reaction diffusion equations (see e.g. \cite[Section 4]{Arnulf1}).

For additive noise, we make the following assumption on the covariance operator.
 \begin{Assumption}
 \label{assumption6}
  We assume  that the covariance operator $Q : H\longrightarrow H$  satisfies 
 \begin{eqnarray*}
 \left\Vert (A(0))^{\frac{\beta-1}{2}}Q^{\frac{1}{2}}\right\Vert_{\mathcal{L}_2(H)}<\infty, 
 \end{eqnarray*}
 where $\beta$ is defined in \assref{assumption1}.
 \end{Assumption}
 In order to achieve convergence order greater than $\frac{1}{2}$ when dealing with additive noise,  we require the following assumption on $F$, which is less restrictive  than those used in \cite{Antjd1,Xiaojie1,Xiaojie2,Antjd2,Antjd3} and hence includes many nonlinear drift functions.
 \begin{Assumption}
 \label{assumption7}
The nonlinear function $F :[0,T]\times H\longrightarrow H$ is differentiable with respect to the second variable and there exists $C_2\geq 0$ such that
 \begin{eqnarray*}
 \Vert F'(t,u)v\Vert\leq C_2\Vert v\Vert,\quad \Vert A^{-\frac{\eta}{2}} \left(F'(t,u)-F'(t,v)\right)\Vert_{L(H)}\leq C\Vert u-v\Vert,\quad t\in[0,T],\; u,v\in H,
 \end{eqnarray*}
 for some $\eta\in (\frac{3}{4}, 1)$, 
 where $ F'(t,u)=\frac{\partial F}{\partial u}(t,u)$ for   $t\in[0,T]$ and $u\in H$.
 \end{Assumption}
 
\begin{theorem}
\label{theorem1}\cite[Theorem 1.3]{Praha} Let Assumptions \ref{assumption2} (i)-(ii),   \ref{assumption1},  \ref{assumption3} and \ref{assumption4}, be fulfilled. Then the non-autonomous problem \eqref{model} has a unique mild solution $X(t)$, which takes the following integral form
\begin{eqnarray}
\label{mild2}
X(t)=U(t,0)X_0+\int_0^tU(t,s)F(s,X(s))ds+\int_0^tU(t,s)B(s,X(s))dW(s),\quad t\in[0, T],
\end{eqnarray} 
where $U(t,s)$ is the evolution system defined in \rmref{remark2}. Moreover, there exists a positive constant $K_5$ such that 
\begin{eqnarray}
\label{borne1}
\sup_{0\leq t\leq T}\Vert X(t)\Vert_{L^2\left(\Omega, \mathcal{D}\left((-A(0))^{\frac{\beta}{2}}\right)\right)}\leq K_5\left(1+\Vert X_0\Vert_{L^2\left(\Omega,\mathcal{D}\left((-A(0))^{\frac{\beta}{2}}\right)\right)}\right).
\end{eqnarray}
\end{theorem} 
\subsection{Numerical scheme and main results}
\label{numerics}
For the seek of simplicity, we consider the family of linear operators $A(t)$ to be of second order and has the following form
\begin{eqnarray}
\label{family}
A(t)u=-\sum_{i,j=1}^d\frac{\partial}{\partial x_i}\left(q_{i,j}(x,t)\frac{\partial u}{\partial x_j}\right)+\sum_{j=1}^dq_j(x,t)\frac{\partial u}{\partial x_j}.
\end{eqnarray}
We require the coefficients $q_{i,j}$ and $q_j$ to be smooth functions of the variable $x\in\overline{\Lambda}$ and H\"{o}lder-continuous with respect to $t\in[0,T]$. We further assume that there exists a positive constant $c$ such that the following  ellipticity condition holds
\begin{eqnarray}
\label{ellip}
\sum_{i,j=1}^dq_{i,j}(x,t)\xi_i\xi_j\geq c\vert \xi\vert^2, \quad (x,t)\in\overline{\Lambda}\times [0,T].
\end{eqnarray}
 Under the above assumptions on $q_{i,j}$ and $q_j$, it is  well known  that  the family of linear operators defined by \eqref{family} fulfils  \assref{assumption2} (i)-(ii), see e.g. \cite[Chapter III, Section 11]{Suzuki}, \cite[Section 7.6]{Pazy} or \cite[Section 5.2]{Tanabe}. The above assumptions on $q_{i,j}$ and $q_j$ also imply that \assref{assumption2} (iii) is fulfilled, see e.g. \cite[Example 6.1]{Praha}, \cite[Chapter III]{Suzuki} or \cite{Amann,Seely}.
 In the abstract form \eqref{model}, the nonlinear functions $F:[0, T]\times H\longrightarrow H$ and $B:[0, T]\times H\longrightarrow HS(Q^{\frac{1}{2}}(H), H)$ are defined by
\begin{eqnarray}
\label{nemistekii1}
(F(t, v))(x)=f(t,x, v(x)),\quad (B(t,v)u)(x)=b(t,x, v(x)).u(x),\quad  x\in \Lambda,\quad  v\in H,
\end{eqnarray}
$u\in Q^{\frac{1}{2}}(H)$,  $t\in[0,T]$, 
  where $f:\mathbb{R}_{+}\times\Lambda\times \mathbb{R}\longrightarrow \mathbb{R}$ and $b:\mathbb{R}_{+}\times\Lambda\times\mathbb{R}\longrightarrow\mathbb{R}$ are continuously differentiable functions with globally bounded derivatives.
As in \cite{Suzuki,Antonio1}, we introduce two spaces $\mathbb{H}$ and $V$, such that $\mathbb{H}\subset V$, that depend on the boundary conditions for the domain of the operator $A(t)$ and the corresponding bilinear form. For example, for Dirichlet  boundary conditions we  introduce the following space
\begin{eqnarray*}
V=H^1_0(\Lambda)=\{v\in H^1(\Lambda) : v=0\quad \text{on}\quad \partial \Lambda\}.
\end{eqnarray*}
For Robin  boundary conditions and  Neumann  boundary conditions, which is a special case of Robin boundary conditions ($\alpha_0=0$), we take $V=H^1(\Lambda)$ and 
\begin{eqnarray*}
\mathbb{H}=\{v\in H^2(\Lambda) : \partial v/\partial v_A+\alpha_0v=0,\quad \text{on}\quad \partial \Lambda\}, \quad \alpha_0\in\mathbb{R}.
\end{eqnarray*}
Using  Green's formula and the boundary conditions, we obtain the associated bilinear form  to $A(t)$ 
\begin{eqnarray*}
a(t)(u,v)=\int_{\Lambda}\left(\sum_{i,j=1}^dq_{ij}(x,t)\dfrac{\partial u}{\partial x_i}\dfrac{\partial v}{\partial x_j}+\sum_{i=1}^dq_i(x,t)\dfrac{\partial u}{\partial x_i}v\right)dx, \quad u,v\in V,
\end{eqnarray*}
for Dirichlet and Neumann boundary conditions and  
\begin{eqnarray*}
a(t)(u,v)=\int_{\Lambda}\left(\sum_{i,j=1}^dq_{ij}(x,t)\dfrac{\partial u}{\partial x_i}\dfrac{\partial v}{\partial x_j}+\sum_{i=1}^dq_i(x,t)\dfrac{\partial u}{\partial x_i}v\right)dx+\int_{\partial\Lambda}\alpha_0uvdx, \quad u,v\in V.
\end{eqnarray*}
for Robin boundary conditions. 
Using  G\aa rding's inequality (see e.g. \cite[(4.3)]{Thomee}) yields
\begin{eqnarray*}
a(t)(v,v)\geq \lambda_0\Vert v \Vert^2_{1}-c_0\Vert v\Vert^2, \quad  v\in V,\quad t\in[0,T].
\end{eqnarray*}
By adding and subtracting $c_{0}X $ on the right hand side of (\ref{model}), we obtain a new family of linear operators that we still denote by  $A(t)$. Therefore the  new corresponding   bilinear form associated to $A(t)$ still denoted by $a(t)$ satisfies the following coercivity property
\begin{eqnarray}
\label{ellip2}
a(t)(v,v)\geq \; \lambda_0\Vert v\Vert_{1}^{2},\;\;\;\;\; v \in V,\quad t\in[0,T].
\end{eqnarray}
Note that the expression of the nonlinear term $F$ has changed as we have  included the term $-c_{0}X$
in the new nonlinear term that we still denoted by $F$.

The coercivity property (\ref{ellip2}) implies that $A(t)$, $t\in[0, T]$ is sectorial on $L^2(\Lambda)$, see e.g., \cite{Stig2}. Therefore   $-A(t)$, $t\in[0, T]$ generates an analytic semigroups denoted  by    $S_t(s)=:e^{-s A(t)}$  on $L^{2}(\Lambda)$  such that \cite{Henry}
\begin{eqnarray*}
S_t(s):= e^{-s A(t)}=\dfrac{1}{2 \pi i}\int_{\mathcal{C}} e^{ s\lambda}(\lambda I - A(t))^{-1}d \lambda,\quad
\;s>0,
\end{eqnarray*}
where $\mathcal{C}$  denotes a path that surrounds the spectrum of $-A(t)$.
The coercivity  property \eqref{ellip2} also implies that $A(t)$ is a positive operator and its fractional powers are well defined and 
  for any $\alpha>0$ we have
\begin{equation}
\label{fractional}
 \left\{\begin{array}{rcl}
         (A(t))^{-\alpha} & =& \frac{1}{\Gamma(\alpha)}\displaystyle\int_0^\infty  s^{\alpha-1}{\rm e}^{-sA(t)}ds,\\
         (A(t))^{\alpha} & = & ((A(t))^{-\alpha})^{-1},
        \end{array}\right.
\end{equation}
where $\Gamma(\alpha)$ is the Gamma function (see \cite{Henry}).  
 The domain  of $(A(t))^{\frac{\alpha}{2}}$  are  characterized in \cite{Suzuki,Stig1,Stig2}  for $1\leq \alpha\leq 2$  with equivalence of norms as follows:
\begin{eqnarray*}
\mathcal{D}((A(t))^{\frac{\alpha}{2}})=H^1_0(\Lambda)\cap H^{\alpha}(\Lambda)\hspace{1cm} 
\text{(for Dirichlet boundary condition)}\nonumber\\
\mathcal{D}(A(t))=\mathbb{H},\quad \mathcal{D}((A(t))^{\frac{1}{2}})=H^1(\Lambda)\hspace{0.5cm} \text{(for Robin boundary condition)}\nonumber\\
\Vert v\Vert_{H^{\alpha}(\Lambda)}\equiv \Vert ((A(t))^{\frac{\alpha}{2}}v\Vert:=\Vert v\Vert_{\alpha,t},\quad  v\in \mathcal{D}((A(t))^{\frac{\alpha}{2}}).\nonumber
\end{eqnarray*}
The characterization of $\mathcal{D}((A(t))^{\frac{\alpha}{2}})$ for $0\leq \alpha<1$ can be found in  \cite[Theorems 2.1 \& 2.2]{Nambu}.

Now, we turn our attention to the discretization of the problem \eqref{model}.
We start by splitting the domain $\Lambda$ in finite triangles. Let $\mathcal{T}_h$ be the triangulation with maximal length $h$ satisfying the usual regularity assumptions, and $V_h\subset V$ be the space of continuous functions that are piecewise linear over the triangulation $\mathcal{T}_h$. We consider the projection $P_h$ from $H=L^2(\Lambda)$ to $V_h$ defined for every $u\in H$ by 
\begin{eqnarray}
\label{proj1}
\langle P_hu, \chi\rangle=\langle u,\chi\rangle, \quad \phi, \chi \in V_h.
\end{eqnarray}
 For all $t\in[0,T]$, the discrete operator $A_h(t) :V_h\longrightarrow V_h$ is defined by
 \begin{eqnarray}
 \label{discreteoper}
 \langle A_h(t)\phi,\chi\rangle=\langle A(t)\phi,\chi\rangle=-a(t)(\phi,\chi), \quad \phi,\chi\in V_h.
 \end{eqnarray}
 The coercivity property (\ref{ellip2}) implies that $A_h(t)$ is sectorial on $L^2(\Lambda)$, see e.g., \cite{Stig2} or \cite[Chapter III, Section 12]{Suzuki}. Therefore   $-A_h(t)$ generates an analytic semi group denoted  by    $S^h_t(s)=:e^{-s A_h(t)}$  on $L^{2}(\Lambda)$ .
 The coercivity property \eqref{ellip2}  also implies that there exist constants $C_2>0$ and $\theta\in(\frac{1}{2}\pi,\pi)$ such that (see e.g., \cite[(2.9)]{Stig2} or \cite{Suzuki,Henry})
 \begin{eqnarray}
 \label{sectorial1}
 \Vert (\lambda\mathbf{I}-A_h(t))^{-1}\Vert_{L(H)}\leq \frac{C_2}{\vert \lambda\vert},\quad \lambda \in S_{\theta}
 \end{eqnarray}
 holds uniformly for $h>0$ and $t\in[0,T]$.
 The coercivity property \eqref{ellip2} also implies that the smooth properties \eqref{smooth}  
 hold for $A_h$, uniformly on $h>0$ and $t\in[0,T]$, i.e. for all $\alpha\geq 0$ and $\gamma\in[0,1]$, 
 there exists a positive constant $C_3$ such that the following estimates hold uniformly on $h>0$ and $t\in[0,T]$, see e.g., \cite{Suzuki,Henry}
 \begin{eqnarray}
 \label{smooth1}
 \left\Vert(A_h(t))^{\alpha}e^{-sA_h(t)}\right\Vert_{L(H)}\leq C_3s^{-\alpha}, \;s>0, \;
 \label{smooth2}
  \left\Vert (A_h(t))^{-\gamma}\left(\mathbf{I}-e^{-rA_h(t)}\right)\right\Vert_{L(H)}\leq C_3r^{\gamma}, r\geq 0.
 \end{eqnarray}
 The semi-discrete version of problem \eqref{model} consists of finding $X^h(t)\in V_h$, such that 
 \begin{eqnarray}
 \label{semi1}
 dX^h(t)+A_h(t)X^h(t)dt=P_hF(t,X^h(t))dt+P_hB(t,X^h(t))dW(t),\quad t\in(0,T],
 \end{eqnarray}
with $X^h(0)=P_hX_0$.

Throughout this paper, we take $t_m=m\Delta t\in[0,T]$, where $\Delta t=\frac{T}{M}$ for a given $ M\in\mathbb{N}$, $m\in\{0,\cdots,M\}$,  $C$ is a generic constant that may change from one place to another. 
Applying the linear implict Euler method  to \eqref{semi1} gives the following fully discrete scheme 
\begin{eqnarray}
\label{implicit1}
\left\{\begin{array}{ll}
X^h_0=P_hX_0, \quad \\
X^h_{m+1}=S^m_{h,\Delta t}X^h_m+\Delta tS^m_{h,\Delta t}P_hF(t_m, X^h_m)+S^m_{h,\Delta t}P_hB(t_m, X^h_m)\Delta W_m,\quad m=0,\cdots
,M-1,
\end{array}
\right.
\end{eqnarray}
where $\Delta W_m$ and $S^m_{h,\Delta t}$ are defined respectively by 
\begin{eqnarray}
\label{operator1}
\Delta W_m :=W({t_{m+1}})-W({t_{m}}), \quad  S^m_{h,\Delta t}:=(\mathbf{I}+\Delta tA_{h,m})^{-1}\quad \text{and}\quad A_{h,m}:=A_h(t_{m}).
\end{eqnarray}
Having the numerical method  \eqref{implicit1}  in hand, our goal is to analyze its strong convergence toward the mild solution in the  $L^2$ norm.
The main results of this paper are formulated in the following theorems. 
\begin{theorem}\textbf{[Multiplicative noise]}
\label{mainresult1}
Let $X(t_m)$ and  $X^h_m$ be respectively the mild solution of \eqref{model} and  the numerical approximation given by \eqref{implicit1} at $t_m=m\Delta t$. 
Let Assumptions \ref{assumption1}, \ref{assumption2},  \ref{assumption3} and \ref{assumption4} be fulfilled.
\begin{itemize}
\item[(i)] If  $0<\beta<1$, then the following error  estimate holds
\begin{eqnarray*}
\Vert X(t_m)-X^h_m\Vert_{L^2(\Omega,H)}\leq C\left(h^{\beta}+\Delta t^{\frac{\beta}{2}}\right).
\end{eqnarray*}
\item[(ii)] If  $\beta=1$, then the following error  estimate holds
\begin{eqnarray*}
\Vert X(t_m)-X^h_m\Vert_{L^2(\Omega,H)}\leq C\left(h+\Delta t^{\frac{1}{2}-\epsilon}\right),
\end{eqnarray*}
where $\epsilon$ is a positive number, small enough.
\item[(iii)] If $1<\beta< 2$ and if \assref{assumption5} is fulfilled, then the following error estimate holds
\begin{eqnarray*}
\Vert X(t_m)-X^h_m\Vert_{L^2(\Omega,H)}\leq C\left(h^{\beta}+\Delta t^{\frac{1}{2}}\right).
\end{eqnarray*}
\end{itemize}
\end{theorem}
\begin{theorem}\textbf{[Additive noise]}
\label{mainresult2}
Let $X(t_m)$ and  $X^h_m$ be respectively the mild solution of \eqref{model}   and  the numerical approximation given by \eqref{implicit1} at $t_m=m\Delta t$. 
If  Assumptions \ref{assumption1}, \ref{assumption2},  \ref{assumption3},  \ref{assumption6} and \ref{assumption7} are fulfilled,
then the following error estimate holds
\begin{eqnarray*}
\Vert X(t_m)-X^h_m\Vert_{L^2(\Omega,H)}\leq C\left(h^{\beta}+\Delta t^{\frac{\beta}{2}-\epsilon}\right),
\end{eqnarray*}
for an arbitrarily small $\epsilon>0$.
\end{theorem}

\section{Proof of the main results}
\label{convergenceproof}
The proofs the main results require some preparatory results. 
\subsection{Preparatory results}
\label{Preparation}
\begin{lemma}\cite{Antjd5} or \cite[Chapter III]{Suzuki}.
\label{lemma2}
Let \assref{assumption2} be fulfilled.
\begin{itemize}
\item[(i)] For any $\gamma\in[0,1]$, the following equivalence of norms holds
\begin{eqnarray}
\label{alb1}
C^{-1}\Vert (A_h(0))^{-\gamma}v\Vert\leq \Vert (A_h(t))^{-\gamma}v\Vert\leq C\Vert (A_h(0))^{-\gamma}v\Vert,\quad v\in V_h,\quad t\in[0,T].\nonumber
\end{eqnarray}
\item[(ii)] For any $\gamma\in[0,1]$, it holds that
\begin{eqnarray}
\label{alb2}
C^{-1}\Vert (A_h(0))^{\gamma}v\Vert\leq \Vert (A_h(t))^{\gamma}v\Vert\leq C\Vert (A_h(0))^{\gamma}v\Vert,\quad v\in V_h,\quad t\in[0,T].
\end{eqnarray}
\item[(iii)] For any $\alpha\in[0,1]$, it holds that
\begin{eqnarray}
\label{alb3}
\Vert (A_{h,k})^{\alpha}P_hv\Vert\leq C\Vert (A_{h,l})^{\alpha}v\Vert\leq C\Vert (A(0))^{\alpha}v\Vert,\quad v\in V_h,\quad 0\leq k,l \leq M-1.
\end{eqnarray}
\item[(iv)] The following estimates holds
\begin{eqnarray}
\label{ref3}
\Vert (A_h(t)-A_h(s))(-A_h(r))^{-1}u^h\Vert&\leq& C\vert t-s\vert\Vert u^h\Vert,\quad r,s,t\in[0,T],\quad u^h\in V_h,\nonumber\\
\label{ref2}
\Vert (-A_h(r))^{-1}\left(A_h(s)-A_h(t)\right)u^h\Vert&\leq& C\vert s-t\vert\Vert u^h\Vert,\quad r,s,t\in[0,T],\quad u^h\in V_h\cap D.\nonumber
\end{eqnarray}
\end{itemize}

\begin{remark}
\label{evolutionremark}
From \lemref{lemma2} and the fact that $\mathcal{D}(A_h(t))=\mathcal{D}(A_h(0))$, it follows from \cite[Chapter III, Section 12]{Suzuki} or \cite[Theorem 6.1, Chapter 5]{Pazy} that there exists a unique evolution system $U_h :\Delta(T)\longrightarrow L(H)$, satisfying \cite[(6.3), Page 149]{Pazy}
\begin{eqnarray}
\label{ref6}
U_h(t,s)=S^h_s(t-s)+\int_s^tS^h_{\tau}(t-\tau)R^h(\tau,s)d\tau,
\end{eqnarray}
where  $R^h(t,s):=\sum\limits_{m=1}^{\infty}R^h_m(t,s)$, with $R^h_m(t,s)$ given by \cite[(6.22), Page 153]{Pazy}
\begin{eqnarray*}
R^h_1(t,s):=(A_h(s)-A_h(t))S^h_s(t-s),\quad R^h_{m+1}:=\int_s^tR^h_1(t,s)R^h_m(\tau,s)d\tau,\quad m\geq 1.
\end{eqnarray*}
 Note also that from \cite[(6.6), Chpater 5, Page 150]{Pazy}, the following identity holds 
\begin{eqnarray}
\label{ref7}
R^h(t,s)=R_1^h(t,s)+\int_s^tR_1^h(t,\tau)R^h(\tau,s)d\tau.
\end{eqnarray} 
The mild solution of the semi-discrete problem \eqref{semi1} can therefore be written as
\begin{eqnarray}
\label{mild4}
X^h(t)=U_h(t,0)P_hX_0+\int_0^tU_h(t,s)P_hF\left(s,X^h(s)\right)ds+\int_0^tU_h(t,s)P_hB\left(s,X^h(s)\right)dW(s).
\end{eqnarray}
\end{remark}
\begin{lemma}\cite[Chapter III]{Suzuki}.
\label{pazylemma}
Under \assref{assumption2}, the  evolution system $U_h(t, s)$  satisfies:
\begin{itemize}
\item[(i)] $U_h(.,s)\in C^1(]s,T]; L(H))$, $0\leq s\leq T$ and 
\begin{eqnarray*}
\frac{\partial U_h}{\partial t}(t,s)=A_h(t)U_h(t,s), \quad 
\Vert A_h(t)U_h(t,s)\Vert_{L(H)}\leq \frac{C}{t-s},\quad 0\leq s<t\leq T.
\end{eqnarray*}
\item[(ii)] $U_h(t,.)v\in C^1([0,t[; H)$, $0<t\leq T$, $v\in\mathcal{D}(A_h(0))$ and 
\begin{eqnarray*}
\frac{\partial U_h}{\partial s}(t,s)v=-U_h(t,s)A_h(s)v,\quad 
 \Vert A_h(t)U_h(t,s)A_h(s)^{-1}\Vert_{L(H)}\leq C, \quad 0\leq s\leq t\leq T.
\end{eqnarray*}
\item[(iii)] For any  $(t,r), (r,s)\in \Delta(T)$ it holds that
\begin{eqnarray*}
U_h(s,s)=\mathbf{I}\quad \text{and} \quad U_h(t,r)U_h(r,s)=U_h(t,s).
\end{eqnarray*}
\end{itemize}
\end{lemma}

\end{lemma}
\begin{lemma} \cite[Chapter III]{Suzuki}, \cite{Pazy} or \cite{Antjd5}
\label{evolutionlemma}
Let  \assref{assumption2} be fulfilled. 
\begin{itemize}
\item[(i)] The following estimate holds
\begin{eqnarray}
\label{reste2}
 \Vert U_h(t,s)\Vert_{L(H)}\leq C,\quad 0\leq s\leq t\leq T.
\end{eqnarray}
\item[(ii)] For any $0\leq\alpha\leq 1$, $0\leq\gamma\leq 1$ and $0\leq s\leq t\leq T$, the following estimates hold
\begin{eqnarray}
\label{ae1}
\Vert (A_h(r))^{\alpha}U_h(t,s)\Vert_{L(H)}&\leq& C(t-s)^{-\alpha},\quad r\in[0,T],\\
\label{ae3}
\Vert U_h(t,s)(A_h(r))^{\alpha}\Vert_{L(H)}&\leq& C(t-s)^{-\alpha},\quad r\in[0,T],\\
\label{ae2}
 \Vert (-A_h(r))^{\alpha}U_h(t,s)(A_h(s))^{-\gamma}\Vert_{L(H)}&\leq& C(t-s)^{\gamma-\alpha}, \quad r\in[0,T].
\end{eqnarray}
\item[(iii)] For any $0\leq s\leq t\leq T$,  the following estimates hold
\begin{eqnarray}
\label{hen1}
\Vert \left(U_h(t,s)-\mathbf{I}\right)(A_h(s))^{-\gamma}\Vert_{L (H)}&\leq& C(t-s)^{\gamma}, \quad 0\leq \gamma\leq 1,\\
\label{hen2}
\Vert \left (A_h(r))^{-\gamma}(U_h(t,s)-\mathbf{I}\right)\Vert_{L (H)}&\leq& C(t-s)^{\gamma}, \quad 0\leq \gamma\leq 1.
\end{eqnarray}
\end{itemize}
\end{lemma}

The following space and time regularity for the mild solution of the semi-discrete problem \eqref{semi1} 
will be useful in our convergence analysis. Their proofs can be found in \cite{Antjd3,Antjd4}.
\begin{lemma}
\label{lemma3}
\begin{itemize}
\item[(1)]
Let Assumptions \ref{assumption1}, \ref{assumption2} (i)-(ii), \ref{assumption3} and \ref{assumption4} be fulfilled.  Let $X^h(t)$
be the mild solution of \eqref{semi1} for multiplicative noise.
\begin{itemize}
\item[(i)] If  $0\leq\beta<1$, then for all $\gamma\in[0,\beta]$ the following estimates hold
\begin{eqnarray}
\label{reduit1}
\Vert (A_h(\tau))^{\frac{\gamma}{2}}X^h(t)\Vert_{L^2(\Omega,H)}\leq C\left(1+\Vert \left(A(0)\right)^{\frac{\gamma}{2}} X_0\Vert_{L^2(\Omega, H)}\right),\; 0\leq t,\tau\leq T,\\
\label{reduit1a}
\Vert X^h(\tau)-X^h(r)\Vert_{L^2(\Omega,H)}\leq C(\tau-r)^{\frac{\beta}{2}}\left(1+\Vert X_0\Vert_{L^2(\Omega, H)}\right),\; 0\leq r\leq \tau\leq T.
\end{eqnarray}
\item[(ii)] If $1\leq\beta<2$ and if in addition \assref{assumption5} is fulfilled, then \eqref{reduit1} holds for any $\gamma\in[0,\beta]$ and \eqref{reduit1a} becomes
\begin{eqnarray*}
\Vert X^h(t_2)-X^h(t_1)\Vert_{L^2(\Omega,H)}&\leq &C(t_2-t_1)^{\frac{1}{2}}\left(1+\Vert X_0\Vert_{L^2(\Omega, H)}\right),\quad 0\leq t_1\leq t_2\leq T.\nonumber
\end{eqnarray*}
\end{itemize}
\item[(2)] Let Assumptions  \ref{assumption1}, \ref{assumption2}, \ref{assumption3}, \ref{assumption6} and \ref{assumption7}
be fulfilled. Let $X^h(t)$ be the mild solution of \eqref{semi1} with additive noise and $\gamma\in[0, \beta)$.
Then the following space and time  regularities hold
\begin{eqnarray*}
\Vert (A_h(\tau))^{\frac{\gamma}{2}}X^h(t)\Vert_{L^2(\Omega,H)}&\leq& C\left(1+\Vert \left(A(0)\right)^{\frac{\gamma}{2}} X_0\Vert_{L^2(\Omega, H)}\right),\quad 0\leq t, \tau\leq T,\\
\Vert X^h(t_2)-X^h(t_1)\Vert_{L^2(\Omega,H)}&\leq &C(t_2-t_1)^{\frac{\min(\beta,1)}{2}}\left(1+\Vert X_0\Vert_{L^2(\Omega, H)}\right),\quad 0\leq t_1\leq t_2\leq T.\nonumber
\end{eqnarray*}
\end{itemize}
\end{lemma}

\begin{corollary}
\label{corollary1}
As a consequence of \lemref{lemma3}, under Assumptions \ref{assumption1}, \ref{assumption2} (i)-(ii), \ref{assumption3} and \ref{assumption4}, it holds that
\begin{eqnarray}
\Vert X^h(t)\Vert_{L^2(\Omega, H)}\leq C,\quad \Vert F(t,X^h(t))\Vert_{L^2(\Omega,H)}\leq C,\quad \Vert B(t,X^h(t))\Vert_{L^2(\Omega,H)}\leq C,\quad t\in[0,T].\nonumber
\end{eqnarray}
\end{corollary}

 \begin{lemma}\textbf{[Space error]}\cite{Antjd2}
 \label{lemma4}
 \begin{itemize}
 \item[(1)]
 Let Assumptions \ref{assumption1}, \ref{assumption2} (i)-(ii), \ref{assumption3} and \ref{assumption4} be fulfilled. Let $X(t)$ and $X^h(t)$
 be respectively the mild solution of \eqref{model} and \eqref{semi1} for multiplicative noise.
 \begin{itemize}
 \item[(i)] If $0\leq \beta<1$, then the following space error estimate holds 
 \begin{eqnarray*}
 \Vert X(t)-X^h(t)\Vert_{L^2(\Omega,H)}\leq Ch^{\beta},\quad 0\leq t\leq T.
 \end{eqnarray*}
 \item[(ii)] If $1\leq\beta<2$ and if in addition \assref{assumption5} is fulfilled, then 
 \begin{eqnarray*}
 \Vert X(t)-X^h(t)\Vert_{L^2(\Omega,H)}\leq Ch^{\beta},\quad 0\leq t\leq T.
 \end{eqnarray*}
 \end{itemize}
 \item[(2)] Let Assumptions  \ref{assumption1}, \ref{assumption2}, \ref{assumption3} and  \ref{assumption6} be fulfilled. 
 Let $X(t)$ and $X^h(t)$ be respectively the mild solution of \eqref{model} and \eqref{semi1} for additive noise. Then the following space error holds
 \begin{eqnarray*}
 \Vert X(t)-X^h(t)\Vert_{L^2(\Omega,H)}\leq Ch^{\beta},\quad 0\leq t\leq T.
 \end{eqnarray*}
 \end{itemize}
 \end{lemma}

  For non commutative operators $H_j$, we introduce the following notation
\begin{eqnarray*}
\prod_{j=l}^kH_j:=\left\{\begin{array}{ll}
H_kH_{k-1}\cdots H_l,\quad \text{if} \quad k\geq l,\\
\mathbf{I},\quad \hspace{2.1cm} \text{if} \quad k<l.
\end{array}
\right.
\end{eqnarray*}

\begin{lemma}\cite{Antjd2}
\label{lemma5}
Let  \assref{assumption2} be fulfilled. Then the following estimate holds
\begin{eqnarray}
\label{comp1}
\left\Vert\left(\prod_{j=l}^me^{-\Delta tA_{h,j}}\right)(A_{h,l})^{\gamma}\right\Vert_{L(H)}\leq Ct_{m+1-l}^{-\gamma}, \quad 0\leq l\leq m\leq M,\quad 0\leq \gamma<1.
\end{eqnarray}
\end{lemma}

\begin{lemma}
\label{lemma6}
Let \assref{assumption2} be fulfilled. Then the following estimate holds
\begin{eqnarray}
\left\Vert(A_{h,k})^{\alpha}(\mathbf{I}+sA_{h,j})^{-n}\right\Vert_{L(H)}\leq C_{\alpha}(ns)^{-\alpha},\quad n>\alpha,\quad s>0,\quad 0\leq j,k\leq M.
\end{eqnarray}
\end{lemma}

\begin{proof}
Due to \assref{assumption2} (iii), the proof follows  the same lines as \cite[(6.6)]{Mizutani}.
\end{proof}

The following lemma will be useful in our convergence analysis.
\begin{lemma}
\label{lemma7}
Let \assref{assumption2} be fulfilled. 
\begin{itemize}
\item[(i)]
For any $\alpha\in[0,1)$, it holds that
\begin{eqnarray}
\label{disop}
\left\Vert (A_{h,k})^{\alpha}\left(\prod_{j=i}^mS^j_{h,\Delta t}\right)\right\Vert_{L(H)}\leq Ct_{m-i+1}^{-\alpha},\quad 0\leq i\leq m\leq M,\quad 0\leq k\leq M.
\end{eqnarray}
\item[(ii)]
For any $\alpha_1,\alpha_2\in[0,1)$, it holds that
\begin{eqnarray*}
\left\Vert (A_{h,k})^{\alpha_1}\left(\prod_{j=i}^mS^j_{h,\Delta t}\right)(A_{h,i})^{-\alpha_2}\right\Vert_{L(H)}\leq Ct_{m-i+1}^{-\alpha_1+\alpha_2},\quad 0\leq i\leq m\leq M,\quad 0\leq k\leq M.
\end{eqnarray*}
\item[(iii)] For any $\alpha_1, \alpha_2\in[0, 1)$ and any $1\leq k,l,j\leq M$, the following estimate holds 
\begin{eqnarray*}
\left\Vert (A_{h,k})^{-\alpha_1}\left(S^j_{h,\Delta t}-S^{j-1}_{h, \Delta t}\right)(A_{h,l})^{-\alpha_2}\right\Vert_{L(H)}\leq C\Delta t^2t_j^{-1+\alpha_1+\alpha_2}\leq C\Delta t^{1+\alpha_1+\alpha_2}.
\end{eqnarray*}
\end{itemize}
\end{lemma}

\begin{proof}  
Note that the proof in the case $i=m$ is straightforward. We only concentrate on the case $i<m$. 
The main idea is to compare the discrete evolution operator in \eqref{disop} with the following frozen operator
\begin{eqnarray*}
\prod_{j=i}^mS^i_{h,\Delta t}=\left(\mathbf{I}+\Delta tA_{h,i}\right)^{-(m-i+1)}.
\end{eqnarray*} 
\begin{itemize}
\item[(i)]
Using Lemmas \ref{lemma2}  and \ref{lemma6}, it holds that
\begin{eqnarray*}
\left\Vert (A_{h,k})^{\alpha}(\mathbf{I}+\Delta tA_{h,i})^{-(m-i+1)}\right\Vert_{L(H)}\leq  Ct_{m-i+1}^{-\alpha}.
\end{eqnarray*}
It remains to estimate $(A_{h,k})^{\alpha}\Delta^h_{m,i}$, where
\begin{eqnarray}
\label{fuj}
\Delta_{m,i}^h:=\prod_{j=i}^mS^j_{h,\Delta t}-(S^i_{h,\Delta t})^{m-i+1}.
\end{eqnarray}
One can easily check that the following resolvent identity holds
\begin{eqnarray}
\label{ident1}
&&(\mathbf{I}+\Delta tA_{h,j+1})^{-1}-(\mathbf{I}+\Delta tA_{h,i})^{-1}\nonumber\\
&=&\Delta t(\mathbf{I}+\Delta tA_{h,j+1})^{-1}(A_{h,i}-A_{h,j+1})(\mathbf{I}+\Delta tA_{h,i})^{-1}.
\end{eqnarray}
Using the telescopic sum, it holds that
\begin{eqnarray}
\label{teles}
\Delta_{m,i}^h&=&\sum_{j=0}^{m-i-1}\left(\prod_{k=j+i+1}^mS^k_{h,\Delta t}\right)(\mathbf{I}+\Delta tA_{h,j+i+1})\left[(\mathbf{I}+\Delta tA_{h,j+i+1})^{-1}-(\mathbf{I}+\Delta tA_{h,i})^{-1}\right]\nonumber\\
&&.(\mathbf{I}+\Delta tA_{h,i})^{-j-1}.
\end{eqnarray}
Substituting the identity \eqref{ident1} in \eqref{teles} and rearranging, we obtain
\begin{eqnarray}
\label{ident2}
\Delta_{m,i}^h&=&\Delta t\sum_{j=0}^{m-i-1}\left(\prod_{k=j+i+1}^mS^k_{h,\Delta t}\right)(A_{h,i}-A_{h,j+i+1})(\mathbf{I}+\Delta tA_{h,i})^{-j-2}\nonumber\\
&=&\Delta t\sum_{j=0}^{m-i-1}\Delta^h_{m,j+i+1}(A_{h,i}-A_{h,j+i+1})(\mathbf{I}+\Delta tA_{h,i})^{-j-2}\nonumber\\
&+&\Delta t\sum_{j=0}^{m-i-1}(\mathbf{I}+\Delta tA_{h,j+i+1})^{-(m-j-i)}(A_{h,i}-A_{h,j+i+1})(\mathbf{I}+\Delta tA_{h,i})^{-j-2}.
\end{eqnarray}
Therefore multiplying both sides of \eqref{ident2} by $(A_{h,k})^{\alpha}$ yields
\begin{eqnarray}
\label{ident3}
&&(A_{h,k})^{\alpha}\Delta_{m,i}^h\nonumber\\
&=&\Delta t\sum_{j=0}^{m-i-1}(A_{h,k})^{\alpha}\Delta^h_{m,j+i+1}(A_{h,i}-A_{h,j+i+1})(\mathbf{I}+\Delta tA_{h,i})^{-j-2}\nonumber\\
&+&\Delta t\sum_{j=0}^{m-i-1}(A_{h,k})^{\alpha}(\mathbf{I}+\Delta tA_{h,j+i+1})^{-(m-j-i)}(A_{h,i}-A_{h,j+i+1})(\mathbf{I}+\Delta tA_{h,i})^{-j-2}.
\end{eqnarray}
Taking the norm in both sides of \eqref{ident3},  using triangle inequality, \lemref{lemma6} and \assref{assumption2} yields
\begin{eqnarray}
\label{ident4a}
&&\Vert (A_{h,k})^{\alpha}\Delta^h_{m,i}\Vert_{L(H)}\nonumber\\
&\leq &C\Delta t\sum_{j=0}^{m-i-1}\Vert (A_{h,k})^{\alpha}\Delta^h_{m,j+i+1}\Vert_{L(H)}\Vert (A_{h,i}-A_{h, j+i+1})(\mathbf{I}+\Delta tA_{h,i})^{-j-2}\Vert_{L(H)}\nonumber\\
&+&C\Delta t\sum_{j=0}^{m-i-1}t_{m-j-i}^{-\alpha}\Vert (A_{h,i}-A_{h, j+i+1})(\mathbf{I}+\Delta tA_{h,i})^{-j-2}\Vert_{L(H)}.
\end{eqnarray}
Employing Lemmas \ref{lemma2} and \ref{lemma6}  yields
\begin{eqnarray}
\label{ident4b}
&&\Vert (A_{h,i}-A_{h, j+i+1})(\mathbf{I}+\Delta tA_{h,i})^{-j-2}\Vert_{L(H)}\nonumber\\
&\leq& \Vert (A_{h,i}-A_{h, j+i+1})(A_h(0))^{-1}\Vert_{L(H)}\Vert A_h(0)(\mathbf{I}+\Delta tA_{h,i})^{-j-1}\Vert_{L(H)}\Vert (\mathbf{I}+\Delta tA_{h,i})^{-1}\Vert_{L(H)}\nonumber\\
&\leq& Ct_{j+1}t_{j+1}^{-1}=C.
\end{eqnarray}
Substituting \eqref{ident4b} in \eqref{ident4a} and using the fact that $C\Delta t\sum\limits_{j=0}^{m-i-1}t_{m-j-1}^{-\alpha}\leq C$ yields
\begin{eqnarray}
\label{ident4}
\Vert (A_{h,k})^{\alpha}\Delta^h_{m,i}\Vert_{L(H)}\leq C+C\Delta t\sum_{j=i+1}^{m}\Vert (A_{h,k})^{\alpha}\Delta^h_{m,j}\Vert_{L(H)}.
\end{eqnarray}
Applying the discrete Gronwall lemma to \eqref{ident4} yields
\begin{eqnarray}
\Vert (A_{h,k})^{\alpha}\Delta^h_{m,i}\Vert_{L(H)}\leq C.
\end{eqnarray}
This completes the proof of (i).

\item[(ii)]
Using Lemmas \ref{lemma2}  and \ref{lemma6}, we obtain
\begin{eqnarray}
\label{ident5}
&&\left\Vert (A_{h,k})^{\alpha_1}(\mathbf{I}+\Delta tA_{h,i})^{-(m-i+1)}(A_{h,i})^{-\alpha_2}\right\Vert_{L(H)}\nonumber\\
&\leq& C\left\Vert (A_{h,i})^{\alpha_1}(\mathbf{I}+\Delta tA_{h,i})^{-(m-i+1)}(A_{h,i})^{-\alpha_2}\right\Vert_{L(H)}\nonumber\\
&=&\left\Vert (A_{h,i})^{\alpha_1-\alpha_2}(\mathbf{I}+\Delta tA_{h,i})^{-(m-i+1)}\right\Vert_{L(H)}\nonumber\\
&\leq & Ct_{m-i+1}^{-\alpha_1+\alpha_2}.
\end{eqnarray}
It remains to estimate $(A_{h,k})^{\alpha_1}\Delta^h_{m,i}(A_{h,i})^{-\alpha_2}$, where $\Delta^h_{m,i}$ is defined by \eqref{fuj}.
From \eqref{ident2}, it holds that
\begin{eqnarray}
\label{ident6}
&&(A_{h,k})^{\alpha_1}\Delta_{m,i}^h(A_{h,i})^{-\alpha_2}\\
&=&\Delta t\sum_{j=0}^{m-i-1}(A_{h,k})^{\alpha_1}\Delta^h_{m,j+i+1}(A_{h,i}-A_{h,j+i+1})(\mathbf{I}+\Delta tA_{h,i})^{-j-2}(A_{h,i})^{-\alpha_2}\nonumber\\
&+&\Delta t\sum_{j=0}^{m-i-1}(A_{h,k})^{\alpha_1}(\mathbf{I}+\Delta tA_{h,j+i+1})^{-(m-j-i)}(A_{h,i}-A_{h,j+i+1})(\mathbf{I}+\Delta tA_{h,i})^{-j-2}(A_{h,i})^{-\alpha_2}.\nonumber
\end{eqnarray}
Taking the norm in both sides of \eqref{ident6}, using the triangle inequality, \lemref{lemma6}, \eqref{ident4b}, \lemref{lemma7} (i) and the fact that $(A_{h,i})^{-\alpha_2}$ is uniformly bounded yields 
\begin{eqnarray}
\label{ident7}
\Vert(A_{h,k})^{\alpha_1}\Delta_{m,i}^h(A_{h,i})^{-\alpha_2}\Vert_{L(H)}
&\leq&C\Delta t\sum_{j=0}^{m-i-1}\Vert(A_{h,k})^{\alpha_1}\Delta^h_{m,j+i+1}\Vert_{L(H)}\nonumber\\
&+&C\Delta t\sum_{j=0}^{m-i-1}\Vert(A_{h,k})^{\alpha_1}(\mathbf{I}+\Delta tA_{h,j+i+1})^{-(m-j-i)}\Vert_{L(H)}\nonumber\\
&\leq& C\Delta t\sum_{j=0}^{m-i-1}t_{m-j-i}^{-\alpha_1}+C\Delta t\sum_{j=0}^{m-i-1}t_{m-j-i}^{-\alpha_1}\nonumber\\
&\leq& C.
\end{eqnarray}
From \eqref{fuj}, employing \eqref{ident5} and \eqref{ident7} yields
\begin{eqnarray*}
&&\left\Vert(A_{h,k})^{\alpha_1}\left(\prod_{j=i}^mS^j_{h,\Delta t}\right)(A_{h,i})^{-\alpha_2}\right\Vert_{L(H)}\nonumber\\
&\leq& \left\Vert (A_{h,k})^{\alpha_1}\Delta^h_{m,i}(A_{h,i})^{-\alpha_2}\right\Vert_{L(H)}+\left\Vert(A_{h,k})^{\alpha_1}(\mathbf{I}+\Delta tA_{h,i})^{-(m-i+1)}\right\Vert_{L(H)}\nonumber\\
&\leq& C+Ct_{m-i+1}^{-\alpha_1+\alpha_2}\leq Ct_{m-i+1}^{-\alpha_1+\alpha_2}.
\end{eqnarray*}
This  completes the proof of (ii).
\item[(iii)] Using the identity \eqref{ident1}, it holds that
\begin{eqnarray}
\label{arrangement1}
 &&(A_{h,k})^{-\alpha_1}\left(S^j_{h,\Delta t}-S^{j-1}_{h, \Delta t}\right)(A_{h,l})^{-\alpha_2}\nonumber\\
 &=&\Delta t(A_{h,k})^{-\alpha_1}(\mathbf{I}+\Delta tA_{h,j})^{-1}(A_{h,j-1}-A_{h,j})(\mathbf{I}+\Delta tA_{h,j-1})^{-1}(A_{h,l})^{-\alpha_2}.
\end{eqnarray}
Taking the norm in both sides of \eqref{arrangement1}, employing Lemmas \ref{lemma2} and \ref{lemma6} yields
\begin{eqnarray}
\label{arrangement2}
 &&\left\Vert(A_{h,k})^{-\alpha_1}\left(S^j_{h,\Delta t}-S^{j-1}_{h, \Delta t}\right)(A_{h,l})^{-\alpha_2}\right\Vert_{L(H)}\nonumber\\
 &\leq&\Delta t\left\Vert(\mathbf{I}+\Delta tA_{h,j})^{-1}(A_{h,j})^{1-\alpha_1-\alpha_2}\right\Vert_{L(H)}\left\Vert (A_{h,j})^{-1}(A_{h,j-1}-A_{h,j})\right\Vert_{L(H)}\nonumber\\
 &&\times\left\Vert(\mathbf{I}+\Delta tA_{h,j-1})^{-1}\right\Vert_{L(H)}\nonumber\\
 &\leq& C\Delta t^{2}t_j^{-1+\alpha_1+\alpha_2}\leq C\Delta t^{1+\alpha_1+\alpha_2}.
\end{eqnarray}
This completes the proof of (iii).
\end{itemize}
\end{proof}

The following lemma will be useful to establish error estimates for deterministic problem.
\begin{lemma}
\label{lemma8}
For any  $ \alpha_1,\alpha_2\in[0,1]$, the following estimates hold
\begin{eqnarray}
\label{val1}
\left\Vert (A_{h,k})^{-\alpha_1}\left(e^{-A_{h,j}\Delta t}-S^j_{h,\Delta t}\right)(A_{h,j})^{-\alpha_2}\right\Vert_{L(H)}\leq C\Delta t^{\alpha_1+\alpha_2},\quad 0\leq j,k\leq M,\\
\label{val2}
\left\Vert (A_{h,k})^{\alpha_1}\left(e^{-A_{h,j}\Delta t}-S^j_{h,\Delta t}\right)(A_{h,j})^{-\alpha_2}\right\Vert_{L(H)}\leq C\Delta t^{-\alpha_1+\alpha_2},\quad 0\leq j,k\leq M.
\end{eqnarray}
\end{lemma}

\begin{proof}
We only prove \eqref{val1} since the proof of \eqref{val2} is similar. 
Let us set 
\begin{eqnarray*}
K^j_{h,\Delta t}:=e^{-A_{h,j}\Delta t}-S^j_{h,\Delta t}.
\end{eqnarray*}
 One can easily check that
\begin{eqnarray}
\label{vend1}
-K^j_{h,\Delta t}&=&\int_0^{\Delta t}\frac{d}{ds}\left((\mathbf{I}+sA_{h,j})^{-1}e^{-(\Delta t-s)A_{h,j}}\right)ds=\int_0^{\Delta t}sA_{h,j}^2(\mathbf{I}+sA_{h,j})^{-2}e^{-(\Delta t-s)A_{h,j}}ds\nonumber\\
&=&\int_0^{\Delta t}sA_{h,j}(\mathbf{I}+sA_{h,j})^{-2}A_{h,j}e^{-(\Delta t-s)A_{h,j}}ds.
\end{eqnarray}
Using \lemref{lemma2}, it holds that
\begin{eqnarray}
\label{vend2}
\left\Vert(A_{h,k})^{-\alpha_1}K^j_{h,\Delta t}(A_{h,j})^{-\alpha_2}\right\Vert_{L(H)}\leq C\left\Vert(A_{h,j})^{-\alpha_1}K^j_{h,\Delta t}(A_{h,j})^{-\alpha_2}\right\Vert_{L(H)}.
\end{eqnarray}
From \eqref{vend1} it holds that
{\small
\begin{eqnarray}
\label{vend3}
-(A_{h,j})^{-\alpha_1}K^j_{h,\Delta t}(A_{h,j})^{-\alpha_2}=\int_0^{\Delta t}sA_{h,j}^{1-\alpha_1}(\mathbf{I}+sA_{h,j})^{-2}A_{h,j}^{1-\alpha_2}e^{-(\Delta t-s)A_{h,j}}ds.
\end{eqnarray}
}
Taking the norm in both sides of \eqref{vend3},  employing \eqref{smooth1} and \lemref{lemma6} yields 
\begin{eqnarray*}
\left\Vert -(A_{h,j})^{-\alpha_1}K^j_{h,\Delta t}(A_{h,j})^{-\alpha_2}\right\Vert_{L(H)}&\leq&\int_0^{\Delta t}s\Vert A_{h,j}^{1-\alpha_1}(\mathbf{I}+sA_{h,j})^{-2}\Vert_{L(H)}\Vert A_{h,j}^{1-\alpha_2}e^{-(\Delta t-s)A_{h,j}}\Vert ds\nonumber\\
&\leq & C\int_0^{\Delta t}ss^{-1+\alpha_1}(\Delta t-s)^{-1+\alpha_2}ds\nonumber\\
&\leq & C\Delta t^{\alpha_1+\alpha_2}.
\end{eqnarray*}
This completes the proof of the lemma.
\end{proof}

\begin{lemma}
\label{lemma9}
For all $\alpha_1,\alpha_2>0$ and $\alpha\in[0,1)$, there exist   $C_{\alpha_1\alpha_2}, C_{\alpha,\alpha_2}\geq 0$ such that
\begin{eqnarray}
\label{son1}
\Delta t\sum_{j=1}^mt_{m-j+1}^{-1+\alpha_1}t_j^{-1+\alpha_2}\leq C_{\alpha_1\alpha_2}t_m^{-1+\alpha_1+\alpha_2},\quad
\label{son2}
\Delta t\sum_{j=1}^mt_{m-j+1}^{-\alpha}t_j^{-1+\alpha_2}\leq C_{\alpha\alpha_2}t_m^{-\alpha+\alpha_2}.
\end{eqnarray}
\end{lemma}
\begin{proof}
The proof of the first estimate of \eqref{son1} follows by comparison with the following integral 
\begin{eqnarray*}
\int_0^t(t-s)^{-1+\alpha_1}s^{-1+\alpha_2}ds.
\end{eqnarray*}
The proof of the second estimate of \eqref{son2} is a consequence of the first one. See also \cite{Stig2}.
\end{proof}

The following lemma will be very important to establish our convergence results.
\begin{lemma}
\label{lemma10}
Let $0\leq \alpha< 2$ and let \assref{assumption2} be fulfilled. 
\begin{itemize}
\item[(i)] If $v\in \mathcal{D}((A(0))^{\frac{\alpha}{2}})$, then the following estimate holds
\begin{eqnarray}
\left\Vert \left(\prod_{j=i}^{m}e^{-A_{h,j}\Delta t}\right)P_hv-\left(\prod_{j=i-1}^{m-1}S_{h,\Delta t}^j\right)P_hv\right\Vert\leq C\Delta t^{\frac{\alpha}{2}}\Vert v\Vert_{\alpha}, \quad 1\leq i\leq m\leq M.\nonumber
\end{eqnarray}
\item[(ii)]
Moreover, for non smooth data, i.e. for $v\in H$, it holds that 
\begin{eqnarray}
\left\Vert \left(\prod_{j=i}^{m}e^{-A_{h,j}\Delta t}\right)P_hv-\left(\prod_{j=i-1}^{m-1}S_{h,\Delta t}^j\right)P_hv\right\Vert\leq C\Delta t^{\frac{\alpha}{2}}t_{m-i}^{-\frac{\alpha}{2}}\Vert v\Vert, \quad 1\leq i< m\leq M.\nonumber
\end{eqnarray}
\item[(iii)]
For any $\alpha_1, \alpha_2\in[0,1)$ such that $\alpha_1\leq\alpha_2$, it holds that 
\begin{eqnarray}
\left\Vert \left[\left(\prod_{j=i}^{m}e^{-A_{h,j}\Delta t}\right)-\left(\prod_{j=i-1}^{m-1}S_{h,\Delta t}^j\right)\right](A_{h,i})^{\alpha_1-\alpha_2}\right\Vert_{L(H)}\leq C\Delta t^{\alpha_2}t_{m-i}^{-\alpha_1}, \quad 1\leq i< m\leq M.\nonumber
\end{eqnarray}
\item[(iv)] For any $\gamma\in\left[0,1\right)$, it holds that 
\begin{eqnarray}
\left\Vert \left[\left(\prod_{j=i}^{m}e^{-A_{h,j}\Delta t}\right)-\left(\prod_{j=i-1}^{m-1}S_{h,\Delta t}^j\right)\right](A_{h,i})^{\frac{\gamma}{2}}\right\Vert_{L(H)}\leq C\Delta t^{\frac{1-\gamma-\epsilon}{2}}t_{m-i}^{\frac{-1-\epsilon}{2}}, \quad 1\leq i< m\leq M.\nonumber
\end{eqnarray}
\end{itemize}
\end{lemma}
\begin{proof}
\begin{itemize}
\item[(i)]
Using the telescopic sum, we obtain
\begin{eqnarray}
\label{teles0}
&&\left(\prod_{j=i}^{m}e^{-A_{h,j}\Delta t}\right)P_hv-\left(\prod_{j=i-1}^{m-1}S_{h,\Delta t}^j\right)P_hv\nonumber\\
&=&\sum_{k=1}^{m-i+1}\left(\prod_{j=i+k}^me^{-A_{h,j}\Delta t}\right)\left(e^{-A_{h,i+k-1}\Delta t}-S^{i+k-2}_{h,\Delta t}\right)\left(\prod_{j=i-1}^{i+k-3}S^j_{h,\Delta t}\right)P_hv.\nonumber
\end{eqnarray}
Writing down explicitly the first and the last terms of  the above identity yields
\begin{eqnarray}
\label{teles1}
&&\left(\prod_{j=i}^{m}e^{-A_{h,j}\Delta t}\right)P_hv-\left(\prod_{j=i-1}^{m-1}S_{h,\Delta t}^j\right)P_hv\nonumber\\
&=&\left(e^{-A_{h,m}\Delta t}-S^{m-1}_{h,\Delta t}\right)\left(\prod_{j=i-1}^{m-2}S^j_{h,\Delta t}\right)P_hv+\left(\prod_{j=i+1}^me^{-A_{h,j}\Delta t}\right)\left(e^{-A_{h,i}\Delta t}-S^{i-1}_{h,\Delta t}\right)P_hv\nonumber\\
&+&\sum_{k=2}^{m-i}\left(\prod_{j=i+k}^me^{-A_{h,j}\Delta t}\right)\left(e^{-A_{h,i+k-1}\Delta t}-S^{i+k-2}_{h,\Delta t}\right)\left(\prod_{j=i-1}^{i+k-3}S^j_{h,\Delta t}\right)P_hv.
\end{eqnarray}
Taking the norm in both sides of \eqref{teles1}, inserting an appropriate power of $A_{h,j}$ and using triangle inequality yields
{\small
\begin{eqnarray}
\label{eti1}
&&\left\Vert\left(\prod_{j=i}^{m}e^{-A_{h,j}\Delta t}\right)P_hv-\left(\prod_{j=i-1}^{m-1}S_{h,\Delta t}^j\right)P_hv\right\Vert\nonumber\\
&\leq&\left\Vert\left(e^{-A_{h,m}\Delta t}-S^{m-1}_{h,\Delta t}\right)(A_{h,m})^{-\frac{\alpha}{2}}(A_{h,m})^{\frac{\alpha}{2}}\left(\prod_{j=i-1}^{m-2}S^j_{h,\Delta t}\right)(A_{h,i})^{-\frac{\alpha}{2}}(A_{h,i})^{\frac{\alpha}{2}}P_hv\right\Vert\nonumber\\
&+&\left\Vert\left(\prod_{j=i+1}^me^{-A_{h,j}\Delta t}\right)\left(e^{-A_{h,i}\Delta t}-S^{i-1}_{h,\Delta t}\right)(A_{h,i})^{-\frac{\alpha}{2}}(A_{h,i})^{\frac{\alpha}{2}}P_hv\right\Vert\nonumber\\
&+&\sum_{k=2}^{m-i}\left\Vert\left(\prod_{j=i+k}^me^{-A_{h,j}\Delta t}\right)(A_{h,i+k})^{1-\epsilon}(A_{h,i+k})^{-1+\epsilon}\left(e^{-A_{h,i+k-1}\Delta t}-S^{i+k-2}_{h,\Delta t}\right)(A_{h,i+k-1})^{-\frac{\alpha}{2}-\epsilon}\right.\nonumber\\
&&.(A_{h,i+k-1})^{\frac{\alpha}{2}+\epsilon}\left.\left(\prod_{j=i-1}^{i+k-3}S^j_{h,\Delta t}\right)(A_{h,i})^{-\frac{\alpha}{2}}(A_{h,i})^{\frac{\alpha}{2}}P_hv\right\Vert\nonumber\\
&=:& I_1+I_2+I_3.
\end{eqnarray}
} 
Using Lemmas \ref{lemma8}, \ref{lemma7} (ii)-(iii) and \ref{lemma2}  yields 
{\small
\begin{eqnarray}
\label{eti2}
&&I_1\nonumber\\
&\leq&\left\Vert\left(e^{-A_{h,m}\Delta t}-S^{m-1}_{h,\Delta t}\right)(A_{h,m})^{-\frac{\alpha}{2}}\right\Vert_{L(H)}\left\Vert(A_{h,m})^{\frac{\alpha}{2}}
\left(\prod_{j=i-1}^{m-2}S^j_{h,\Delta t}\right)(A_{h,i})^{-\frac{\alpha}{2}}\right\Vert_{L(H)}\nonumber\\
&&\times\Vert(A_{h,i})^{\frac{\alpha}{2}}P_hv\Vert\nonumber\\
&\leq&C\left\Vert\left(e^{-A_{h,m}\Delta t}-S^{m-1}_{h,\Delta t}\right)(A_{h,m})^{-\frac{\alpha}{2}}\right\Vert_{L(H)}\Vert v\Vert_{\alpha}\nonumber\\
&\leq&C\left\Vert\left(e^{-A_{h,m}\Delta t}-S^{m}_{h,\Delta t}\right)(A_{h,m})^{-\frac{\alpha}{2}}\right\Vert_{L(H)}\Vert v\Vert_{\alpha}+C\left\Vert\left(S^{m}_{h, \Delta t}-S^{m-1}_{h,\Delta t}\right)(A_{h,m})^{-\frac{\alpha}{2}}\right\Vert_{L(H)}\Vert v\Vert_{\alpha}\nonumber\\
&\leq & C\Delta t^{\frac{\alpha}{2}}\Vert v\Vert_{\alpha}.
\end{eqnarray}
}
Using Lemmas \ref{lemma5}, \ref{lemma8}, \ref{lemma7}  and \ref{lemma2}  yields 
{\small
\begin{eqnarray}
\label{eti3}
I_2&\leq&\left\Vert\left(\prod_{j=i+1}^me^{-A_{h,j}\Delta t}\right)\right\Vert_{L(H)}\left\Vert\left(e^{-A_{h,i}\Delta t}-S^{i-1}_{h,\Delta t}\right)(A_{h,i})^{-\frac{\alpha}{2}}\right\Vert_{L(H)}\Vert(A_{h,i})^{\frac{\alpha}{2}}P_hv\Vert\nonumber\\
&\leq&C\left\Vert\left(e^{-A_{h,i}\Delta t}-S^{i-1}_{h,\Delta t}\right)(A_{h,i})^{-\frac{\alpha}{2}}\right\Vert_{L(H)}\Vert v\Vert_{\alpha}\nonumber\\
&\leq&C\left\Vert\left(e^{-A_{h,i}\Delta t}-S^{i}_{h,\Delta t}\right)(A_{h,i})^{-\frac{\alpha}{2}}\right\Vert_{L(H)}\Vert v\Vert_{\alpha}+C\left\Vert\left(S^{i}_{h, \Delta t}-S^{i-1}_{h,\Delta t}\right)(A_{h,i})^{-\frac{\alpha}{2}}\right\Vert_{L(H)}\Vert v\Vert_{\alpha}\nonumber\\
&\leq & C\Delta t^{\frac{\alpha}{2}}\Vert v\Vert_{\alpha}.
\end{eqnarray}
}
Using Lemmas \ref{lemma5}, \ref{lemma8},  \ref{lemma7}, \ref{lemma2} and \ref{lemma9} as in the estimate of $I_1$ and $I_2$ yields 
\begin{eqnarray}
\label{eti4}
I_3&\leq&\sum_{k=2}^{m-i}\left\Vert\left(\prod_{j=i+k}^me^{-A_{h,j}\Delta t}\right)(A_{h,i+k})^{1-\epsilon}\right\Vert_{L(H)}\nonumber\\
&&\times\left\Vert(A_{h,i+k})^{-1+\epsilon}\left(e^{-A_{h,i+k-1}\Delta t}-S^{i+k-2}_{h,\Delta t}\right)(A_{h,i+k-1})^{-\frac{\alpha}{2}-\epsilon}\right\Vert_{L(H)}\nonumber\\
&&\times\left\Vert(A_{h,i+k-1})^{\frac{\alpha}{2}+\epsilon}\left(\prod_{j=i-1}^{i+k-3}S^j_{h,\Delta t}\right)(A_{h,i})^{-\frac{\alpha}{2}}\right\Vert_{L(H)}\Vert(A_{h,i})^{\frac{\alpha}{2}}P_hv\Vert\nonumber\\
&\leq & C\sum_{k=2}^{m-i}t_{m+1-i-k}^{-1+\epsilon}\Delta t^{1+\frac{\alpha}{2}}t_{k-1}^{-\epsilon}= C\Delta t^{\frac{\alpha}{2}}\sum_{k=2}^{m-i}t_{m-i-k+1}^{-1+\epsilon}t_{k-1}^{-\epsilon}\Delta t\nonumber\\
&\leq & C\Delta t^{\frac{\alpha}{2}}.
\end{eqnarray}
Substituting  \eqref{eti4}, \eqref{eti3} and \eqref{eti2} in \eqref{eti1} yields
{\small
\begin{eqnarray}
&&\left\Vert\left(\prod_{j=i}^{m}e^{-A_{h,j}\Delta t}\right)P_hv-\left(\prod_{j=i-1}^{m-1}S_{h,\Delta t}^j\right)P_hv\right\Vert\leq C\Delta t^{\frac{\alpha}{2}}\Vert v\Vert_{\alpha}.
\end{eqnarray}
}
This completes the proof of (i).
\item[(ii)] For non smooth initial data,
taking the norm in both sides of \eqref{teles1} and inserting an appropriate power of $A_{h,j}$ yields
{\small
\begin{eqnarray}
\label{teles2}
&&\left\Vert \left(\prod_{j=i}^{m}e^{-A_{h,j}\Delta t}\right)P_hv-\left(\prod_{j=i-1}^{m-1}S_{h,\Delta t}^j\right)P_hv\right\Vert\nonumber\\
&\leq& \left\Vert\left(e^{-A_{h,m}\Delta t}-S^{m-1}_{h,\Delta t}\right)(A_{h,m})^{-\frac{\alpha}{2}}\right\Vert_{L(H)}\left\Vert(A_{h,m})^{\frac{\alpha}{2}}
\left(\prod_{j=i-1}^{m-2}S^j_{h,\Delta t}\right)P_hv\right\Vert\nonumber\\
&+&\left\Vert\left(\prod_{j=i+1}^me^{-A_{h,\Delta t}}\right)(A_{h,i+1})^{\frac{\alpha}{2}}\right\Vert_{L(H)}\left\Vert(A_{h,i+1})^{-\frac{\alpha}{2}}\left(e^{-A_{h,i}\Delta t}-S^{i-1}_{h,\Delta t}\right)P_hv\right\Vert\nonumber\\
&+&\sum_{k=2}^{m-i}\left\Vert\left(\prod_{j=i+k}^me^{-A_{h,j}\Delta t}\right)(A_{h,i+k})^{1-\epsilon}\right\Vert_{L(H)}\nonumber\\
&&\times\left\Vert(A_{h,i+k})^{-1+\epsilon}\left(e^{-A_{h,i+k-1}\Delta t}-S^{i+k-2}_{h,\Delta t}\right)(A_{h,i+k-1})^{-1+\epsilon}\right\Vert_{L(H)}\nonumber\\
&&\times\left\Vert(A_{h,i+k-1})^{1-\epsilon}\left(\prod_{j=i-1}^{i+k-3}S^j_{h,\Delta t}\right)P_hv\right\Vert.
\end{eqnarray}
}
Employing Lemmas \ref{lemma8}, \ref{lemma7} (i), \ref{lemma5} and \ref{lemma9}, it follows from \eqref{teles2} that
\begin{eqnarray}
\label{teles3}
&&\left\Vert \left(\prod_{j=i}^{m}e^{-A_{h,j}\Delta t}\right)P_hv-\left(\prod_{j=i-1}^{m-1}S_{h,\Delta t}^j\right)P_hv\right\Vert\nonumber\\
&\leq&C\Delta t^{\frac{\alpha}{2}}t_{m-i}^{-\frac{\alpha}{2}}\Vert v\Vert+C\Delta t^{\frac{\alpha}{2}}t_{m-i}^{-\frac{\alpha}{2}}\Vert v\Vert+C\Delta t^{1-2\epsilon}\sum_{k=2}^{m-i}\Delta tt_{m-i-k+1}^{-1+\epsilon}t_{k-1}^{-1+\epsilon}\Vert v\Vert\nonumber\\
&\leq & C\Delta t^{\frac{\alpha}{2}}t_{m-i-k}^{-\frac{\alpha}{2}}\Vert v\Vert+C\Delta t^{\frac{\alpha}{2}}t_{m-i}^{-\frac{\alpha}{2}}\Vert v\Vert+ C\Delta t^{1-2\epsilon}t_{m-i}^{-1+2\epsilon}\Vert v\Vert\nonumber\\
&\leq & C\Delta t^{\frac{\alpha}{2}} t_{m-i}^{-\frac{\alpha}{2}}.
\end{eqnarray}
\item[(iii)]  Inserting an appropriate power of $A_{h,i}$ in \eqref{teles1} and taking the norm in both sides yields
{\small
\begin{eqnarray}
\label{tala2}
&&\left\Vert \left[\left(\prod_{j=i}^{m}e^{-A_{h,j}\Delta t}\right)-\left(\prod_{j=i-1}^{m-1}S_{h,\Delta t}^j\right)\right](A_{h,i})^{\alpha_1-\alpha_2}\right\Vert_{L(H)}\nonumber\\
&\leq& \left\Vert\left(e^{-A_{h,m}\Delta t}-S^{m-1}_{h,\Delta t}\right)(A_{h,m})^{-\alpha_2}\right\Vert_{L(H)}\left\Vert(A_{h,m})^{\alpha_2}
\left(\prod_{j=i-1}^{m-2}S^j_{h,\Delta t}\right)(A_{h,i})^{\alpha_1-\alpha_2}\right\Vert_{L(H)}\nonumber\\
&+&\left\Vert\left(\prod_{j=i+1}^me^{-A_{h,\Delta t}}\right)(A_{h,i+1})^{\alpha_1}\right\Vert_{L(H)}\left\Vert(A_{h,i+1})^{-\alpha_1}\left(e^{-A_{h,i}\Delta t}-S^{i-1}_{h,\Delta t}\right)(A_{h,i})^{-(\alpha_2-\alpha_1)}\right\Vert_{L(H)}\nonumber\\
&+&\sum_{k=2}^{m-i}\left\Vert\left(\prod_{j=i+k}^me^{-A_{h,j}\Delta t}\right)(A_{h,i+k})^{\alpha_2+\epsilon}\right\Vert_{L(H)}\nonumber\\
&&\times\left\Vert(A_{h,i+k})^{-\alpha_2-\epsilon}\left(e^{-A_{h,i+k-1}\Delta t}-S^{i+k-2}_{h,\Delta t}\right)(A_{h,i+k-1})^{-1+\epsilon}\right\Vert_{L(H)}\nonumber\\
&&\times\left\Vert(A_{h,i+k-1})^{1-\epsilon}\left(\prod_{j=i-1}^{i+k-3}S^j_{h,\Delta t}\right)(A_{h,i})^{-(\alpha_2-\alpha_1)}\right\Vert_{L(H)}.
\end{eqnarray}
}
Employing  Lemmas \ref{lemma8}, \ref{lemma7},  \ref{lemma5} and \ref{lemma9},  it follows that
\begin{eqnarray}
\label{tala3}
&&\left\Vert \left[\left(\prod_{j=i}^{m}e^{-A_{h,j}\Delta t}\right)-\left(\prod_{j=i-1}^{m-1}S_{h,\Delta t}^j\right)\right](A_{h,i})^{\alpha_1-\alpha_2}\right\Vert_{L(H)}\nonumber\\
&\leq&C\Delta t^{\alpha_2}t_{m-i}^{-\alpha_1}+C\Delta t^{\alpha_2}t_{m-i}^{-\alpha_1}+C\Delta t^{\alpha_2}\sum_{k=2}^{m-i}\Delta tt_{m-i-k+1}^{-\alpha_2-\epsilon}t_{k-1}^{-1+\epsilon+\alpha_2-\alpha_1}\nonumber\\
&\leq & C\Delta t^{\alpha_2}t_{m-i-k}^{-\alpha_1}+C\Delta t^{\alpha_2}t_{m-i}^{-\alpha_1}+ C\Delta t^{\alpha_2}t_{m-i}^{-\alpha_1}\nonumber\\
&\leq& C\Delta t^{\alpha_2}t_{m-i}^{-\alpha_1}.
\end{eqnarray}
This completes the proof of (iii).
\item[(iv)]
Inserting an appropriate power of $A_{h,j}$ in \eqref{teles1}, taking the norm in both sides  and using triangle inequality yields
\begin{eqnarray}
\label{revi1}
&&\left\Vert\left(\prod_{j=i}^{m}e^{-A_{h,j}\Delta t}\right)P_hv-\left(\prod_{j=i-1}^{m-1}S_{h,\Delta t}^j\right)(A_{h,0})^{\frac{\gamma}{2}} P_hv\right\Vert\nonumber\\
&\leq&\left\Vert\left(e^{-A_{h,m}\Delta t}-S^{m-1}_{h,\Delta t}\right)(A_{h,0})^{\frac{-(1-\gamma-\epsilon)}{2}}(A_{h,0})^{\frac{1-\gamma-\epsilon}{2}}\left(\prod_{j=i-1}^{m-2}S^j_{h,\Delta t}\right)(A_{h,0})^{\frac{\gamma}{2}}P_hv\right\Vert\nonumber\\
&+&\left\Vert\left(\prod_{j=i+1}^me^{-A_{h,j}\Delta t}\right)(A_{h,0})^{\frac{1-\epsilon}{2}}(A_{h,0})^{\frac{-1+\epsilon}{2}}\left(e^{-A_{h,i}\Delta t}-S^{i-1}_{h,\Delta t}\right)(A_{h,0})^{\frac{\gamma}{2}}P_hv\right\Vert\nonumber\\
&+&\sum_{k=2}^{m-i}\left\Vert\left(\prod_{j=i+k}^me^{-A_{h,j}\Delta t}\right)(A_{h,0})^{1-\epsilon}(A_{h,0})^{-1+\epsilon}\left(e^{-A_{h,0}\Delta t}-S^{i+k-2}_{h,\Delta t}\right)(A_{h,0})^{\frac{-(1-\gamma-\epsilon)}{2}}\right.\nonumber\\
&&.(A_{h,0})^{\frac{1-\gamma-\epsilon}{2}}\left.\left(\prod_{j=i-1}^{i+k-3}S^j_{h,\Delta t}\right)(A_{h,0})^{\frac{\gamma}{2}}P_hv\right\Vert\nonumber\\
&=:& J_1+J_2+J_3.
\end{eqnarray}
Using Lemmas \ref{lemma8}, \ref{lemma7} (ii)-(iii) and \ref{lemma2}  yields 
\begin{eqnarray}
\label{revi2}
&&J_1\nonumber\\
&\leq&\left\Vert\left(e^{-A_{h,0}\Delta t}-S^{m-1}_{h,\Delta t}\right)(A_{h,0})^{\frac{-(1-\gamma-\epsilon)}{2}}\right\Vert_{L(H)}\left\Vert(A_{h,0})^{\frac{1-\gamma-\epsilon}{2}}
\left(\prod_{j=i-1}^{m-2}S^j_{h,\Delta t}\right)(A_{h,0})^{\frac{\gamma}{2}}P_hv\right\Vert_{L(H)}\nonumber\\
&\leq&C\left\Vert\left(e^{-A_{h,m}\Delta t}-S^{m-1}_{h,\Delta t}\right)(A_{h,0})^{\frac{-(1-\gamma-\epsilon)}{2}}\right\Vert_{L(H)} t_{m-i}^{\frac{-1+\epsilon}{2}}\Vert v\Vert\nonumber\\
&\leq & C\Delta t^{\frac{1-\gamma-\epsilon}{2}}t_{m-i}^{\frac{-1+\epsilon}{2}}\Vert v\Vert.
\end{eqnarray}
Using Lemmas \ref{lemma5}, \ref{lemma8}, \ref{lemma7}  and \ref{lemma2}  yields 
{\small
\begin{eqnarray}
\label{revi3}
J_2&\leq&\left\Vert\left(\prod_{j=i+1}^me^{-A_{h,j}\Delta t}\right)(A_{h,0})^{\frac{1-\epsilon}{2}}\right\Vert_{L(H)}\left\Vert(A_{h,0})^{\frac{-1+\epsilon}{2}}\left(e^{-A_{h,i}\Delta t}-S^{i-1}_{h,\Delta t}\right)(A_{h,0})^{\frac{\gamma}{2}}P_hv\right\Vert_{L(H)}\nonumber\\
&\leq&\left\Vert\left(\prod_{j=i+1}^me^{-A_{h,j}\Delta t}\right)(A_{h,0})^{\frac{1-\epsilon}{2}}\right\Vert_{L(H)}\left\Vert(A_{h,0})^{\frac{-1+\epsilon}{2}}\left(e^{-A_{h,i}\Delta t}-S^{i-1}_{h,\Delta t}\right)(A_{h,0})^{\frac{\gamma}{2}}\right\Vert_{L(H)}\Vert P_hv\Vert\nonumber\\
&\leq & C\Delta t^{\frac{1-\gamma-\epsilon}{2}}t_{m-i}^{\frac{-1+\epsilon}{2}}\Vert v\Vert.
\end{eqnarray}
}
Using Lemmas \ref{lemma5}, \ref{lemma8},  \ref{lemma7}, \ref{lemma2} and \ref{lemma9} as in the estimate of $J_1$ and $J_2$ yields 
\begin{eqnarray}
\label{revi4}
J_3&\leq&\sum_{k=2}^{m-i}\left\Vert\left(\prod_{j=i+k}^me^{-A_{h,j}\Delta t}\right)(A_{h,0})^{1-\epsilon}\right\Vert_{L(H)}\nonumber\\
&&\times\left\Vert(A_{h,0})^{-1+\epsilon}\left(e^{-A_{h,0}\Delta t}-S^{i+k-2}_{h,\Delta t}\right)(A_{h,0})^{\frac{-(1-\gamma-\epsilon)}{2}}\right\Vert_{L(H)}\nonumber\\
&&\times\left\Vert(A_{h,0})^{\frac{1-\gamma-\epsilon}{2}}\left(\prod_{j=i-1}^{i+k-3}S^j_{h,\Delta t}\right)(A_{h,0})^{\frac{\gamma}{2}}P_hv\right\Vert_{L(H)}\nonumber\\
&\leq & C\sum_{k=2}^{m-i}t_{m+1-i-k}^{-1+\epsilon}\Delta t^{1-\epsilon}t_{k-1}^{\frac{-1-\epsilon}{2}}\Delta t^{\frac{1-\gamma-\epsilon}{2}}\Vert v\Vert= C\Delta t^{\frac{1-\gamma-3\epsilon}{2}}\sum_{k=2}^{m-i}t_{m-i-k+1}^{-1+\epsilon}t_{k-1}^{\frac{-1-\epsilon}{2}}\Delta t\Vert v\Vert\nonumber\\
&\leq & C\Delta t^{\frac{1-\gamma-\epsilon}{2}}t_{m-i}^{\frac{-1+\epsilon}{2}}\Vert v\Vert.
\end{eqnarray}
Substituting  \eqref{revi4}, \eqref{revi3} and \eqref{revi2} in \eqref{revi1} yields
\begin{eqnarray}
\left\Vert\left(\prod_{j=i}^{m}e^{-A_{h,j}\Delta t}\right)P_hv-\left(\prod_{j=i-1}^{m-1}S_{h,\Delta t}^j\right)(A_{h,0})^{\frac{\gamma}{2}} P_hv\right\Vert\leq C\Delta t^{\frac{1-\gamma-\epsilon}{2}}t_{m-i}^{\frac{-1+\epsilon}{2}}\Vert v\Vert.
\end{eqnarray}
\end{itemize}
\end{proof}

\begin{remark}
\label{remark3}
\lemref{lemma10} (i)-(ii)  generalizes \cite[Theorem 7.7 \& Theorem 7.8]{Thomee} (for constant and self-adjoint operator $A(t)=A$
to the case of not necessary self-adjoint and time dependent linear operator $A(t)$.
\end{remark}
\begin{lemma}\cite{Antjd4}
\label{lemma11}
\begin{itemize}
\item[(i)] Let \assref{assumption6} be fulfilled. Then the following estimate holds
\begin{eqnarray*}
\left\Vert (A_h(t))^{\frac{\beta-1}{2}}P_hQ^{\frac{1}{2}}\right\Vert_{\mathcal{L}_2(H)}\leq C,\quad t\in[0,T],
\end{eqnarray*}
where $\beta$ is the parameter defined in \assref{assumption1}.
\item[(ii)] Under \assref{assumption7}, the following estimates hold
\begin{eqnarray*}
\Vert P_hF'(t,u) v\Vert\leq C\Vert v\Vert,\; \Vert A_h^{-\frac{\eta}{2}}P_h \left(F'(t,u)-F'(t,v)\right)\Vert_{L(H)}\leq C\Vert u-v\Vert,
\end{eqnarray*}
for all $t\in[0,T],\; u,v\in H$, where $\eta$ comes from \assref{assumption7} and $C$ is independent of $u$,  $v$, $t$ and $h$. 
\end{itemize}
\end{lemma}
The following lemma is useful in our convergence analysis. 
\begin{lemma}
\label{Mainmulti}
 For all $1\leq i\leq m\leq M$.
 For all $\alpha\in[0, 1)$, the following estimate holds
\begin{eqnarray}
\label{fonda000}
\left\Vert\left[\left(\prod_{j=i}^mU_h(t_j,t_{j-1})\right)-\left(\prod_{j=i-1}^{m-1}
e^{\Delta tA_{h,j}}\right)\right](-A_{h,i-1})^{\alpha}\right\Vert_{L(H)}\leq C\Delta t^{1-\alpha-\epsilon}t_{m-i+1}^{-\alpha+\epsilon},\nonumber
\end{eqnarray}
for an arbitrarily small $\epsilon>0$.
\end{lemma}

The following lemma will be useful
\begin{lemma}
\label{reviewimportant}
Let $0\leq \alpha< 2$ and let \assref{assumption2} be fulfilled. 
\begin{itemize}
\item[(i)] If $v\in \mathcal{D}((A(0))^{\frac{\alpha}{2}})$, then the following estimate holds
\begin{eqnarray}
\left\Vert \left(\prod_{j=i}^{m}U_h(t_j, t_{j-1})\right)P_hv-\left(\prod_{j=i-1}^{m-1}S_{h,\Delta t}^j\right)P_hv\right\Vert\leq C\Delta t^{\frac{\alpha}{2}}\Vert v\Vert_{\alpha}, \quad 1\leq i\leq m\leq M.\nonumber
\end{eqnarray}
\item[(ii)]
Moreover, for non smooth data, i.e. for $v\in H$, it holds that 
\begin{eqnarray}
\left\Vert \left(\prod_{j=i}^{m}U_h(t_j, t_{j-1})\right)P_hv-\left(\prod_{j=i-1}^{m-1}S_{h,\Delta t}^j\right)P_hv\right\Vert\leq C\Delta t^{\frac{\alpha}{2}}t_{m-i}^{-\frac{\alpha}{2}}\Vert v\Vert, \quad 1\leq i< m\leq M.\nonumber
\end{eqnarray}
\item[(iii)]
For any $\alpha_1, \alpha_2\in[0,1)$ such that $\alpha_1\leq\alpha_2$, it holds that 
\begin{eqnarray}
\left\Vert \left[\left(\prod_{j=i}^{m}U_h(t_j, t_{j-1})\right)-\left(\prod_{j=i-1}^{m-1}S_{h,\Delta t}^j\right)\right](A_{h,i})^{\alpha_1-\alpha_2}\right\Vert_{L(H)}\leq C\Delta t^{\alpha_2}t_{m-i}^{-\alpha_1}, \quad 1\leq i< m\leq M.\nonumber
\end{eqnarray}
\item[(iv)] For any $\gamma\in\left[0,1\right)$, it holds that 
\begin{eqnarray}
\left\Vert \left[\left(\prod_{j=i}^{m}U_h(t_j, t_{j-1})\right)-\left(\prod_{j=i-1}^{m-1}S_{h,\Delta t}^j\right)\right](A_{h,i})^{\frac{\gamma}{2}}\right\Vert_{L(H)}\leq C\Delta t^{\frac{1-\gamma-\epsilon}{2}}t_{m-i}^{\frac{-1-\epsilon}{2}}, \quad 1\leq i< m\leq M.\nonumber
\end{eqnarray}
\end{itemize}
\end{lemma}
\begin{proof}
We only prove (i) since the proofs of (ii)--(iv) are similar. Adding and subtracting terms yields the following decomposition
\begin{eqnarray}
\label{merci1}
&&\left(\prod_{j=i}^{m}U_h(t_j, t_{j-1})\right)P_hv-\left(\prod_{j=i-1}^{m-1}S_{h,\Delta t}^j\right)P_hv\nonumber\\
&=&\left[\left(\prod_{j=i}^{m}U_h(t_j, t_{j-1})\right)-\left(\prod_{j=i-1}^{m-1}e^{\Delta tA_{h,j}}\right)\right]P_hv+\left[\left(\prod_{j=i-1}^{m-1}e^{\Delta tA_{h,j}}\right)-\left(\prod_{j=i-1}^{m}e^{\Delta tA_{h,j}}\right)\right]P_hv\nonumber\\
&+&\left[\left(\prod_{j=i-1}^{m}e^{\Delta tA_{h,j}}\right)-\left(\prod_{j=i}^{m}e^{\Delta tA_{h,j}}\right)\right]P_hv+\left[\left(\prod_{j=i}^{m}e^{\Delta tA_{h,j}})\right)-\left(\prod_{j=i-1}^{m-1}S_{h,\Delta t}^j\right)\right]P_hv\nonumber\\
&=:&K_1+K_2+K_3+K_4.
\end{eqnarray}
Using \lemref{Mainmulti} with $\alpha=0$ yields
\begin{eqnarray}
\label{merci2}
\Vert K_1\Vert_{L(H)}\leq \left\Vert\left(\prod_{j=i}^{m}U_h(t_j, t_{j-1})\right)-\left(\prod_{j=i-1}^{m-1}e^{\Delta tA_{h,j}}\right)\right\Vert_{L(H)}\Vert P_hv\Vert\leq C\Delta t^{1-\epsilon}\Vert v\Vert_{\alpha}.
\end{eqnarray}
Using  \lemref{lemma5}  yields
\begin{eqnarray}
\label{merci3}
\Vert K_2\Vert_{L(H)}&\leq& \left\Vert\left(\mathbf{I}-e^{\Delta tA_{h,m}}\right)\left(\prod_{j=i-1}^{m-1}e^{\Delta tA_{h,j}}\right)P_hv\right\Vert_{L(H)}\nonumber\\
&\leq& \left\Vert\left(\mathbf{I}-e^{\Delta tA_{h,m}}\right)A_{h,m}^{-\frac{\alpha}{2}}\right\Vert_{L(H)}\left\Vert A_{h,m}^{\frac{\alpha}{2}}\left(\prod_{j=i-1}^{m-1}e^{\Delta tA_{h,j}}\right)A_{h,i-1}^{-\frac{\alpha}{2}}\right\Vert_{L(H)}\left\Vert  A_{h,i-1}^{\frac{\alpha}{2}} P_hv\right\Vert_{L(H)}\nonumber\\
&\leq& C\Delta t^{\frac{\alpha}{2}}\Vert v\Vert_{\alpha}.
\end{eqnarray}
The term $K_3$ is very similar to  $K_2$. Hence along the same lines as \eqref{merci3}, one easily get
\begin{eqnarray}
\label{merci4}
\Vert K_3\Vert_{L(H)}\leq C\Delta t^{\frac{\alpha}{2}}. 
\end{eqnarray}
Employing \lemref{lemma10} yields
\begin{eqnarray}
\label{merci5}
\Vert K_4\Vert_{L(H)}&\leq& \left\Vert \left[\left(\prod_{j=i}^{m}e^{\Delta tA_{h,j}})\right)-\left(\prod_{j=i-1}^{m-1}S_{h,\Delta t}^j\right)\right]P_hv\right\Vert_{L(H)}\leq C\Delta t^{\frac{\alpha}{2}}.
\end{eqnarray}
Substituting \eqref{merci5}, \eqref{merci4}, \eqref{merci3} and \eqref{merci2} in \eqref{merci1} completes the proof of (i).
\end{proof}

With the above preparatory results, we are now ready to prove our main results.
\subsection{Proof of \thmref{mainresult1}}
\label{mainproof}
Iterating the numerical solution  \eqref{implicit1} at $t_m$ by substituting $X^h_j$, $j=1,2,\cdots,m-1$ only in the first term of \eqref{implicit1} 
by their expressions, we obtain
\begin{eqnarray}
\label{num1}
X^h_m&=&\left(\prod_{j=0}^{m-1}S_{h,\Delta t}^j\right)P_hX_0+\Delta tS^{m-1}_{h,\Delta t}P_hF\left(t_{m-1}, X^h_{m-1}\right)\nonumber\\
&+&S^{m-1}_{h,\Delta t}P_hB\left(t_{m-1}, X^h_{m-1}\right)\Delta W_{m-1}+\Delta t\sum_{i=2}^{m}\left(\prod_{j=m-i}^{m-1}S^j_{h,\Delta t}\right)P_hF\left(t_{m-i}, X^h_{m-i}\right)\nonumber\\
&+&\sum_{i=2}^{m}\left(\prod_{j=m-i}^{m-1}S^j_{h,\Delta t}\right)P_hB\left(t_{m-i}, X^h_{m-i}\right)\Delta W_{m-i}.
\end{eqnarray}
Rewritten the numerical approximation \eqref{num1} in the integral form yields
\begin{eqnarray}
\label{num1a}
X^h_m&=&\left(\prod_{j=0}^{m-1}S_{h,\Delta t}^j\right)P_hX_0+\int_{t_{m-1}}^{t_m}S^{m-1}_{h,\Delta t}P_hF\left(t_{m-1}, X^h_{m-1}\right)ds\nonumber\\
&+&\int_{t_{m-1}}^{t_m}S^{m-1}_{h,\Delta t}P_hB\left(t_{m-1}, X^h_{m-1}\right)dW(s)\nonumber\\
&+&\sum_{i=2}^{m}\int_{t_{m-i}}^{t_{m-i+1}}\left(\prod_{j=m-i}^{m-1}S^j_{h,\Delta t}\right)P_hF\left(t_{m-i}, X^h_{m-i}\right)ds\nonumber\\
&+&\sum_{i=2}^{m}\int_{t_{m-i}}^{t_{m-i+1}}\left(\prod_{j=m-i}^{m-1}S^j_{h,\Delta t}\right)P_hB\left(t_{m-i}, X^h_{m-i}\right)dW(s).
\end{eqnarray}
 Note that the mild solution of \eqref{semi1} can be written as follows:
\begin{eqnarray}
\label{mild5}
X^h(t_m)&=&U_h(t_m,t_{m-1})X^h(t_{m-1})+\int_{t_{m-1}}^{t_m}U_h(t_m,s)P_hF\left(s,X^h(s)\right)ds\nonumber\\
&+&\int_{t_{m-1}}^{t_m}U_h(t_m,s)P_hB\left(s,X^h(s)\right)dW(s).
\end{eqnarray}
Iterating the mild solution \eqref{mild5} yields
\begin{eqnarray}
\label{num2}
X^h(t_m)&=&\left(\prod_{j=1}^mU_h(t_j, t_{j-1})\right)P_hX_0+\int_{t_{m-1}}^{t_m}U_h(t_m,s)P_hF(s,X^h(s))ds\nonumber\\
&+&\int_{t_{m-1}}^{t_m}U_h(t_m,s)P_hB(s,X^h(s))dW(s)\nonumber\\
&+&\sum_{k=1}^{m-1}\int_{t_{m-k-1}}^{t_{m-k}}\left(\prod_{j=m-k+1}^{m}U_h(t_j,t_{j-1})\right)U_h(t_{m-k},s)P_hF(s, X^h(s))ds\nonumber\\
&+&\sum_{k=1}^{m-1}\int_{t_{m-k-1}}^{t_{m-k}}\left(\prod_{j=m-k+1}^{m}U_h(t_j,t_{j-1})\right)U_h(t_{m-k},s)P_hB(s,X^h(s))dW(s).
\end{eqnarray}
Subtracting \eqref{num2} from \eqref{num1a}, taking the $L^2$ norm and using   triangle inequality yields
\begin{eqnarray}
\label{num3}
\left\Vert X^h(t_m)-X^h_m\right\Vert^2_{L^2(\Omega, H)}\leq 25 \sum_{i=0}^4\Vert II_i\Vert^2_{L^2(\Omega,H)},
\end{eqnarray}
where
\begin{eqnarray*}
\label{num4}
II_0&=&\left(\prod_{j=1}^{m}U_h(t_j, t_{j-1})\right)P_hX_0-\left(\prod_{j=0}^{m-1}S_{h,\Delta t}^j\right)P_hX_0,\\
II_1&=&\int_{t_{m-1}}^{t_m}\left[U_h(t_m,s)P_hF\left(s,X^h(s)\right)-S^{m-1}_{h,\Delta t}P_hF\left(t_{m-1},X^h_{m-1}\right)\right]ds, \nonumber\\
II_2&=&\int_{t_{m-1}}^{t_m}\left[U_h(t_m,s)P_hB\left(s, X^h(s)\right)-S^{m-1}_{h,\Delta t}P_hB\left(t_{m-1}, X^h_{m-1}\right)\right]dW(s),\nonumber\\
II_3&=& \sum_{i=2}^{m}\int_{t_{m-i}}^{t_{m-i+1}}\left(\prod_{j=m-i+2}^{m}U_h(t_j, t_{j-1})\right)U_h(t_{m-i+1}, s)P_hF\left(s, X^h(s)\right)ds\nonumber\\
&&-\sum_{i=2}^{m}\int_{t_{m-i}}^{t_{m-i+1}}\left(\prod_{j=m-i}^{m-1}S^j_{h,\Delta t}\right)P_hF\left(t_{m-i},X^h_{m-i}\right)ds,\\
II_4&=&\sum_{i=2}^{m}\int_{t_{m-i}}^{t_{m-i+1}}\left(\prod_{j=m-i+2}^{m}U_h(t_j, t_{j-1})\right)U_h(t_{m-i+1},s)P_hB\left(s,X^h(s)\right)dW(s)\nonumber\\
&&-\sum_{i=2}^{m}\int_{t_{m-i}}^{t_{m-i+1}}\left(\prod_{j=m-i}^{m-1}S^j_{h,\Delta t}\right)P_hB\left(t_{m-i}, X^h_{m-i}\right)dW(s).\nonumber
\end{eqnarray*}
In the following sections, we estimate $II_i$, $i=0,\cdots,4$ separately.
\subsubsection{Estimate of $II_0$, $II_1$ and $II_2$}
Using  \lemref{reviewimportant}, it holds that
\begin{eqnarray}
\label{deb1}
\Vert II_{0}\Vert_{L^2(\Omega, H)}&\leq& \left\Vert \left(\prod_{j=1}^{m}U_h(t_j, t_{j-1}\right)P_hX_0-\left(\prod_{j=0}^{m-1}S^j_{h,\Delta t}\right)P_hX_0\right\Vert_{L^2(\Omega, H)}\nonumber\\
&\leq& C\Delta t^{\frac{\beta}{2}}\Vert A_h^{\frac{\beta}{2}}P_hX_0\Vert_{L^2(\Omega, H)}\leq C\Delta t^{\frac{\beta}{2}}.
\end{eqnarray}
The term $II_1$ can be recast in three terms as follows:
\begin{eqnarray}
\label{mar1}
II_1&=&\int_{t_{m-1}}^{t_m}U_h(t_m,s)\left[P_hF\left(s,X^h(s)\right)-P_hF\left(t_{m-1},X^h(t_{m-1})\right)\right]ds\nonumber\\
&+&\int_{t_{m-1}}^{t_m}\left[U_h(t_m,s)-S^{m-1}_{h,\Delta t}\right]P_hF\left(t_{m-1},X^h(t_{m-1})\right) ds\nonumber\\
&+&\int_{t_{m-1}}^{t_m}S^{m-1}_{h,\Delta t}\left[P_hF\left(t_{m-1},X^h(t_{m-1})\right)-P_hF\left(t_{m-1},X^h_{m-1}\right)\right]ds\nonumber\\
&:=&II_{11}+II_{12}+II_{13}.
\end{eqnarray}
Therefore using triangle inequality yields
\begin{eqnarray}
\label{mar2}
\Vert II_1\Vert_{L^2(\Omega,H)}\leq \Vert II_{11}\Vert_{L^2(\Omega,H)}+\Vert II_{12}\Vert_{L^2(\Omega,H)}+\Vert II_{13}\Vert_{L^2(\Omega,H)}.
\end{eqnarray}
Using triangle inequality \lemref{evolutionlemma} and \coref{corollary1}, it holds that
\begin{eqnarray}
\label{mar3}
\Vert II_{11}\Vert_{L^2(\Omega,H)}&\leq& C\int_{t_{m-1}}^{t_m}\Vert P_hF\left(s,X^h(s)\right)\Vert_{L^2(\Omega,H)}ds\nonumber\\
&+&C\int_{t_{m-1}}^{t_m}\Vert P_hF\left(t_{m-1}, X^h(t_{m-1})\right)\Vert_{L^2(\Omega,H)}ds\nonumber\\
&\leq& C\int_{t_{m-1}}^{t_m}ds\leq C\Delta t.
\end{eqnarray}
Using triangle inequality, Lemmas \ref{evolutionlemma}, \ref{lemma7} and \coref{corollary1}, it holds that
\begin{eqnarray}
\label{mar4}
\Vert II_{12}\Vert_{L^2(\Omega,H)}\leq C\int_{t_{m-1}}^{t_m}\Vert P_hF\left(t_{m-1}, X^h(t_{m-1})\right)\Vert_{L^2(\Omega,H)}ds
\leq C\int_{t_{m-1}}^{t_m}ds\leq C\Delta t.
\end{eqnarray}
Using \lemref{lemma7} (i) with $\alpha=0$ and \assref{assumption3}, it holds that
\begin{eqnarray}
\label{mar5}
\Vert II_{13}\Vert_{L^2(\Omega,H)}\leq C\Delta t\Vert X^h(t_{m-1})-X^h_{m-1}\Vert_{L^2(\Omega,H)}.
\end{eqnarray}
Substituting \eqref{mar5}, \eqref{mar4} and \eqref{mar3} in \eqref{mar2} yields 
\begin{eqnarray}
\label{mar6}
\Vert II_1\Vert_{L^2(\Omega,H)}\leq C\Delta t+C\Delta t\Vert X^h(t_{m-1})-X^h_{m-1}\Vert_{L^2(\Omega,H)}.
\end{eqnarray}
We recast $II_2$ in three terms as follows:
\begin{eqnarray}
\label{isen1}
II_2&=&\int_{t_{m-1}}^{t_m}U_h(t_m,s)\left[P_hB\left(s,X^h(s)\right)-P_hB\left(t_{m-1},X^h(t_{m-1})\right)\right]dW(s)\nonumber\\
&+&\int_{t_{m-1}}^{t_m}\left[U_h(t_m,s)-S^{m-1}_{h,\Delta t}\right]P_hB\left(t_{m-1},X^h(t_{m-1})\right) dW(s)\nonumber\\
&+&\int_{t_{m-1}}^{t_m}S^{m-1}_{h,\Delta t}\left[P_hB\left(t_{m-1},X^h(t_{m-1})\right)-P_hB\left(t_{m-1},X^h_{m-1}\right)\right]dW(s)\nonumber\\
&:=&II_{21}+II_{22}+II_{23}.
\end{eqnarray}
Using triangle inequality and the inequality $(a+b+c)^2\leq 9a^2+9b^2+9c^2$, $a,b,c\in\mathbb{R}$, yields
\begin{eqnarray}
\label{isen2}
\Vert II_2\Vert^2_{L^2(\Omega,H)}\leq 9\Vert II_{21}\Vert^2_{L^2(\Omega,H)}+9\Vert II_{22}\Vert^2_{L^2(\Omega,H)}+9\Vert II_{23}\Vert^2_{L^2(\Omega,H)}.
\end{eqnarray}
Using the It\^{o} isometry, \lemref{evolutionlemma}, \assref{assumption4} and \lemref{lemma3}, it holds that
\begin{eqnarray}
\label{isen3}
\Vert II_{21}\Vert^2_{L^2(\Omega,H)}&=&\int_{t_{m-1}}^{t_m}\left\Vert U_h(t_m,s)\left[P_hB\left(s,X^h(s)\right)-P_hB\left(t_{m-1},X^h(t_{m-1})\right)\right]\right\Vert^2_{L^2(\Omega,H)}ds\nonumber\\
&\leq& C\int_{t_{m-1}}^{t_m}(s-t_{m-1})^{\min(\beta,1)}ds\leq C\Delta t^{\min(\beta+1,2)}.
\end{eqnarray}
Employing the It\^{o} isometry, Lemmas \ref{evolutionlemma}, \ref{lemma7} and \coref{corollary1}, it holds that
\begin{eqnarray}
\label{isen4}
\Vert II_{22}\Vert^2_{L^2(\Omega,H)}&=&\int_{t_{m-1}}^{t_m}\left\Vert\left[U_h(t_m,s)-S^{m-1}_{h,\Delta t}\right]P_hB\left(t_{m-1},X^h(t_{m-1})\right)\right\Vert^2_{L^2(\Omega,H)}ds\nonumber\\
&\leq& C\int_{t_{m-1}}^{t_m}ds\leq C\Delta t.
\end{eqnarray}
Employing It\^{o} isometry,  Lemmas \ref{lemma7} (i) with $\alpha=0$ and \assref{assumption4} yields
\begin{eqnarray}
\label{isen5}
\Vert II_{23}\Vert^2_{L^2(\Omega,H)}&=& \int_{t_{m-1}}^{t_m}\left\Vert S^{m-1}_{h,\Delta t}\left[P_hB\left(t_{m-1},X^h(t_{m-1})\right)-P_hB\left(t_{m-1},X^h_{m-1}\right)\right]\right\Vert^2_{L^2(\Omega,H)}ds\nonumber\\
&\leq & C\Delta t\Vert X^h(t_{m-1})-X^h_{m-1}\Vert^2_{L^2(\Omega,H)}.
\end{eqnarray}
Substituting \eqref{isen5}, \eqref{isen4} and \eqref{isen3} in \eqref{isen2} yields 
\begin{eqnarray}
\label{isen6}
\Vert II_2\Vert^2_{L^2(\Omega,H)}\leq C\Delta t+C\Delta t\Vert X^h(t_{m-1})-X^h_{m-1}\Vert^2_{L^2(\Omega,H)}.
\end{eqnarray}

\subsubsection{Estimate of $II_3$}
We can recast $II_3$ in four terms as follows:
\begin{eqnarray}
&&II_3\nonumber\\
&=&\sum_{i=2}^{m}\int_{t_{m-i}}^{t_{m-i+1}}\left(\prod_{j=m-i+2}^{m}U_h(t_j, t_{j-1})\right)\left[U_h(t_{m-i+1}, s)-U_h(t_{m-i+1}, t_{m-i})\right] P_hF\left(s,X^h(s)\right)ds\nonumber\\
&+&\sum_{i=2}^{m}\int_{t_{m-i}}^{t_{m-i+1}}\left(\prod_{j=m-i+1}^{m}U_h(t_j, t_{j-1})\right)\left[P_hF\left(s,X^h(s)\right)-P_hF\left(t_{m-i+1}, X^h(t_{m-i})\right)\right]ds\nonumber\\
&+&\sum_{i=2}^{m}\int_{t_{m-i}}^{t_{m-i+1}}\left[\left(\prod_{j=m-i+1}^{m}U_h(t_j, t_{j-1})\right)-\left(\prod_{j=m-i}^{m-1}S^j_{h,\Delta t}\right)\right]P_hF\left(t_{m-i+1}, X^h(t_{m-i})\right)ds\nonumber\\
&+&\sum_{i=2}^{m}\int_{t_{m-i}}^{t_{m-i+1}}\left(\prod_{j=m-i}^{m-1}S^j_{h,\Delta t}\right)\left[P_hF\left(t_{m-i+1},X^h(t_{m-i})\right)-P_hF\left(t_{m-i}, X^h_{m-i}\right)\right]ds\nonumber\\
&:=& II_{31}+II_{32}+II_{33}+II_{34}.
\end{eqnarray}
Therefore, employing the triangle inequality yields
\begin{eqnarray}
\label{cont0}
\Vert II_3\Vert_{L^2(\Omega,H)}\leq \Vert II_{31}\Vert_{L^2(\Omega,H)}+\Vert II_{32}\Vert_{L^2(\Omega,H)}+\Vert II_{33}\Vert_{L^2(\Omega,H)}+\Vert II_{34}\Vert_{L^2(\Omega,H)}.
\end{eqnarray}
Inserting an appropriate power of $A_{h,m-i}$, using Lemmas \ref{evolutionlemma}, \ref{pazylemma} (iii) and \coref{corollary1} yields
\begin{eqnarray}
\label{cont1}
\Vert II_{31}\Vert_{L^2(\Omega,H)}&\leq& \sum_{i=2}^{m}\int_{t_{m-i}}^{t_{m-i+2}}\left\Vert\left(\prod_{j=m-i+1}^{m}U_h(t_j, t_{j-1})\right)(A_{h,m-i})^{1-\epsilon}\right\Vert_{L(H)}\nonumber\\
&&\times\Vert (A_{h,m-i})^{-1+\epsilon} U_h(t_{m-i+1}, s)(A_{h,m-i})^{1-\epsilon}\Vert_{L(H)}\nonumber\\
&& \times \left\Vert (A_{h,m-i})^{-1+\epsilon}\left(\mathbf{I}-U_h(s, t_{m-i})\right)\right\Vert_{L(H)}\Vert P_hF\left(s,X^h(s)\right)\Vert_{L^2(\Omega,H)}ds\nonumber\\
&\leq& C\sum_{i=2}^{m}\int_{t_{m-i}}^{t_{m-i+1}}\left\Vert U_h(t_m, t_{m-i})(A_{h,m-i})^{1-\epsilon}\right\Vert_{L(H)}(s-t_{m-i})^{1-\epsilon}ds\nonumber\\
&\leq& C\Delta t^{1-\epsilon}\sum_{i=2}^m\int_{t_{m-i}}^{t_{m-i+1}}t_i^{-1+\epsilon}ds\leq C\Delta t^{1-\epsilon}\sum_{i=2}^{m}\Delta t t_{i}^{-1+\epsilon}\leq C\Delta t^{1-\epsilon}.
\end{eqnarray}
Using triangle inequality, Lemmas \ref{pazylemma} (iii), \ref{evolutionlemma}, \assref{assumption3} and \lemref{lemma3} yields 
\begin{eqnarray}
\label{cont2}
&&\Vert II_{32}\Vert_{L^2(\Omega,H)}\nonumber\\
&\leq &\sum_{i=2}^{m}\int_{t_{m-i}}^{t_{m-i+1}}\left\Vert\left(\prod_{j=m-i+1}^{m}U_h(t_j, t_{j-1})\right)\right\Vert_{L(H)}\nonumber\\
&&\times \Vert P_hF\left(s,X^h(s)\right)-P_hF\left(t_{m-i+1},X^h(t_{m-i})\right)\Vert_{L^2(\Omega,H)}ds\nonumber\\
&\leq& C\sum_{i=2}^{m}\int_{t_{m-i}}^{t_{m-i+1}}\Vert U_h(t_m, t_{m-i})\Vert_{L(H)}\left[(t_{m-i+1}-s)^{\frac{\beta}{2}}+\Vert X^h(s)-X^h(t_{m-i})\Vert_{L^2(\Omega,H)}\right]ds\nonumber\\
&\leq& C\sum_{i=2}^{m}\int_{t_{m-i}}^{t_{m-i+1}}\left[(t_{m-i+1}-s)^{\frac{\beta}{2}}+(s-t_{m-i})^{\frac{\min(\beta,1)}{2}}\right] ds\leq C\Delta t^{\frac{\min(\beta,1)}{2}}.
\end{eqnarray}
Using triangle inequality, \lemref{reviewimportant},  \coref{corollary1} and the fact that $\beta<2$, it holds that
\begin{eqnarray}
\label{cont3}
\Vert II_{33}\Vert_{L^2(\Omega,H)}&\leq& \sum_{i=2}^{m}\int_{t_{m-i}}^{t_{m-i+1}}\left\Vert\left(\prod_{j=m-i+1}^{m}U_h(t_j, t_{j-1})\right)-\left(\prod_{j=m-i}^{m-1}S^j_{h,\Delta t}\right)\right\Vert_{L(H)}\nonumber\\
&&\times \Vert P_hF\left(t_{m-i+1},X^h(t_{m-i})\right)\Vert_{L^2(\Omega,H)}ds\nonumber\\
&\leq& C\sum_{i=2}^{m}\int_{t_{m-i}}^{t_{m-i+1}}t_{i-1}^{-\frac{\beta}{2}}\Delta t^{\frac{\beta}{2}}ds\leq C\Delta t^{\frac{\beta}{2}}.
\end{eqnarray}
Using \lemref{lemma7} (i) with $\alpha=0$ and \assref{assumption3} yields 
\begin{eqnarray}
\label{cont4}
\Vert II_{34}\Vert_{L^2(\Omega,H)}&\leq& \sum_{i=2}^{m}\int_{t_{m-i}}^{t_{m-i+1}}\left\Vert\left(\prod_{j=m-i}^{m-1}S^j_{h,\Delta t}\right)\right\Vert_{L(H)}\nonumber\\
&&\times\left\Vert\left[P_hF\left(t_{m-i+1},X^h(t_{m-i})\right)-P_hF\left(t_{m-i}, X^h_{m-i}\right)\right]\right\Vert_{L^2(\Omega,H)}ds\nonumber\\
&\leq& C\Delta t^{\frac{\beta}{2}}+C\Delta t\sum_{i=2}^{m}\Vert X^h(t_{m-i})-X^h_{m-i}\Vert_{L^2(\Omega,H)}.
\end{eqnarray}
Substituting \eqref{cont4}, \eqref{cont3}, \eqref{cont2} and \eqref{cont1} in \eqref{cont0} yields 
\begin{eqnarray}
\label{cont5}
\Vert II_3\Vert_{L^2(\Omega,H)}\leq C\Delta t^{\frac{\min(\beta,1)}{2}}+C\Delta t\sum_{i=2}^{m}\Vert X^h(t_{m-i})-X^h_{m-i}\Vert_{L^2(\Omega,H)}.
\end{eqnarray}

\subsubsection{Estimate of $II_4$}
We recast  $II_4$  in four terms as follows:
{\small
\begin{eqnarray}
\label{pa1}
&&II_4\nonumber\\
&=&\sum_{i=2}^{m}\int_{t_{m-i}}^{t_{m-i+1}}\left(\prod_{j=m-i+2}^{m}U_h(t_j, t_{j-1})\right)\left[U_h(t_{m-i+1}, s)-U_h(t_{m-i+1}, t_{m-i})\right] P_hB\left(s,X^h(s)\right)dW(s)\nonumber\\
&+&\sum_{i=2}^{m}\int_{t_{m-i}}^{t_{m-i+1}}\left(\prod_{j=m-i+1}^{m}U_h(t_j, t_{j-1})\right)\left[P_hB\left(s,X^h(s)\right)-P_hB\left(t_{m-i+1}, X^h(t_{m-i})\right)\right]dW(s)\nonumber\\
&+&\sum_{i=2}^{m}\int_{t_{m-i}}^{t_{m-i+1}}\left[\left(\prod_{j=m-i+1}^{m}U_h(t_j, t_{j-1})\right)-\left(\prod_{j=m-i}^{m-1}S^j_{h,\Delta t}\right)\right]P_hB\left(t_{m-i+1}, X^h(t_{m-i})\right)dW(s)\nonumber\\
&+&\sum_{i=2}^{m}\int_{t_{m-i}}^{t_{m-i+1}}\left(\prod_{j=m-i}^{m-1}S^j_{h,\Delta t}\right)\left[P_hB\left(t_{m-i+1},X^h(t_{m-i})\right)-P_hB\left(t_{m-i}, X^h_{m-i}\right)\right]dW(s)\nonumber\\
&:=& II_{41}+II_{42}+II_{43}+II_{44}.
\end{eqnarray}
}
Therefore using triangle inequality, we obtain 
\begin{eqnarray}
\label{pa2rev}
\Vert II_4\Vert^2_{L^2(\Omega,H)}\leq 16\Vert II_{41}\Vert^2_{L^2(\Omega,H)}+16\Vert II_{42}\Vert^2_{L^2(\Omega,H)}+16\Vert II_{43}\Vert^2_{L^2(\Omega,H)}+16\Vert II_{44}\Vert^2_{L^2(\Omega,H)}.
\end{eqnarray}
Using the It\^{o} isometry, inserting an appropriate power of $A_{h,m-i}$, using  Lemmas \ref{evolutionlemma}, \ref{pazylemma} (iii)  and \coref{corollary1} yields 
{\small
\begin{eqnarray}
\label{pa3}
&&\Vert II_{41}\Vert^2_{L^2(\Omega,H)}\nonumber\\
&=&\sum_{i=2}^{m}\int_{t_{m-i}}^{t_{m-i+1}}\left\Vert\left(\prod_{j=m-i+2}^{m}U_h(t_j, t_{j-1})\right)\left[U_h(t_{m-i+1}, s)\left(\mathbf{I}-U_h(s, t_{m-i})\right)\right] P_hB\left(s,X^h(s)\right)\right\Vert^2_{L^0_2}ds\nonumber\\
&\leq& C\sum_{i=2}^{m}\int_{t_{m-i}}^{t_{m-i+1}}\left\Vert\left(\prod_{j=m-i+2}^{m}U_h(t_j, t_{j-1})\right)(A_{h,m-i})^{\frac{1-\epsilon}{2}}\right\Vert^2_{L(H)}\nonumber\\
&&\times\Vert(A_{h,m-i})^{\frac{-1+\epsilon}{2}} U_h(t_{m-i+1}, s)(A_{h,m-i})^{\frac{1-\epsilon}{2}}\Vert_{L(H)}\nonumber\\
&& \times \left\Vert (A_{h,m-i})^{\frac{-1+\epsilon}{2}}\left(\mathbf{I}-U_h(s,t_{m-i})\right)\right\Vert^2_{L(H)}\Vert P_hB\left(s,X^h(s)\right)\Vert^2_{L^0_2}ds\nonumber\\
&\leq& C\sum_{i=2}^{m}\int_{t_{m-i}}^{t_{m-i+1}}\left\Vert U_h(t_m, t_{m-i+1})(A_{h,m-i})^{\frac{1-\epsilon}{2}}\right\Vert^2_{L(H)}(s-t_{m-i})^{1-\epsilon}ds\nonumber\\
&\leq& C\sum_{i=2}^{m}\int_{t_{m-i}}^{t_{m-i+1}}t_{i-1}^{-1+\epsilon}(s-t_{m-i})^{1-\epsilon}ds\leq C\Delta t^{1-\epsilon}\sum_{i=2}^{m}\Delta t t_{i-1}^{-1+\epsilon}\leq C\Delta t^{1-\epsilon}.
\end{eqnarray}
}
Using again the It\^{o} isometry, employing \lemref{pazylemma} (iii), \assref{assumption4}, Lemmas \ref{lemma3} and \ref{evolutionlemma} yields 
\begin{eqnarray}
\label{pa4}
&&\Vert II_{42}\Vert^2_{L^2(\Omega,H)}\nonumber\\
&=&\sum_{i=2}^{m}\int_{t_{m-i}}^{t_{m-i+1}}\left\Vert\left(\prod_{j=m-i+1}^{m}U_h(t_j, t_{j-1})\right)\left[P_hB\left(s,X^h(s)\right)-P_hB\left(t_{m-i+1}, X^h(t_{m-i})\right)\right]\right\Vert^2_{L^0_2}ds\nonumber\\
&\leq &\sum_{i=2}^{m}\int_{t_{m-i}}^{t_{m-i+1}}\left\Vert U_h(t_m, t_{m-i})\right\Vert^2_{L(H)} \left\Vert P_hB\left(s,X^h(s)\right)-P_hB\left(t_{m-i+1},X^h(t_{m-i})\right)\right\Vert^2_{L^0_2}ds\nonumber\\
&\leq& C\sum_{i=2}^{m}\int_{t_{m-i}}^{t_{m-i+1}}\left[(t_{m-i+1}-s)^{\beta}+\left\Vert X^h(s)-X^h(t_{m-i})\right\Vert^2_{L^2(\Omega,H)}\right]ds\nonumber\\
&\leq& C\sum_{i=2}^{m}\int_{t_{m-i}}^{t_{m-i+1}}\left[(t_{m-i+1}-s)^{\beta}+(s-t_{m-i})^{\min(\beta, 1)}\right]ds\leq C\Delta t^{\min(\beta,1)}.
\end{eqnarray}
Using the It\^{o} isometry, \lemref{reviewimportant}  and \coref{corollary1}, it holds that
{\small
\begin{eqnarray}
\label{pa5}
&&\Vert II_{43}\Vert^2_{L^2(\Omega,H)}\nonumber\\
&=&\sum_{i=2}^{m}\int_{t_{m-i}}^{t_{m-i+1}}\left\Vert\left[\left(\prod_{j=m-i+1}^{m}U_h(t_j, t_{j-1})\right)-\left(\prod_{j=m-i}^{m-1}S^j_{h,\Delta t}\right)\right]P_hB\left(t_{m-i+1}, X^h(t_{m-i})\right)\right\Vert^2_{L^0_2}ds\nonumber\\
&\leq& \sum_{i=2}^{m}\int_{t_{m-i}}^{t_{m-i+1}}\left\Vert\left(\prod_{j=m-i+1}^{m}U_h(t_j, t_{j-1})\right)-\left(\prod_{j=m-i}^{m-1}S^j_{h,\Delta t}\right)\right\Vert^2_{L(H)}\left\Vert P_hB\left(t_{m-i+1},X^h(t_{m-i})\right)\right\Vert^2_{L^0_2}ds\nonumber\\
&\leq& C\sum_{i=2}^{m}\int_{t_{m-i}}^{t_{m-i+1}}t_{i-1}^{1-\epsilon}\Delta t^{1-\epsilon}ds\leq C\Delta t^{1-\epsilon}.
\end{eqnarray}
}
Using the It\^{o} isometry,  \lemref{lemma7} (i) with $\alpha=0$ and \assref{assumption4} yields 
{\small
\begin{eqnarray}
\label{pa6}
&&\Vert II_{44}\Vert^2_{L^2(\Omega,H)}\nonumber\\
&=&\sum_{i=2}^{m}\int_{t_{m-i}}^{t_{m-i+1}}\left\Vert\left(\prod_{j=m-i}^{m-1}S^j_{h,\Delta t}\right)\left[P_hB\left(t_{m-i+1},X^h(t_{m-i})\right)-P_hB\left(t_{m-i+1}, X^h_{m-i}\right)\right]\right\Vert^2_{L^0_2}ds\nonumber\\
&\leq& \sum_{i=2}^{m}\int_{t_{m-i}}^{t_{m-i+1}}\left\Vert\left(\prod_{j=m-i}^{m-1}S^j_{h,\Delta t}\right)\right\Vert^2_{L(H)}\left\Vert P_hB\left(t_{m-i+1},X^h(t_{m-i})\right)-P_hB\left(t_{m-i+1}, X^h_{m-i}\right)\right\Vert^2_{L^0_2}ds\nonumber\\
&\leq& C\Delta t\sum_{i=2}^{m}\left\Vert X^h(t_{m-i})-X^h_{m-i}\right\Vert^2_{L^2(\Omega,H)}=C\Delta t\sum_{i=0}^{m-2}\Vert X^h(t_i)-X^h_i\Vert^2_{L^2(\Omega, H)}.
\end{eqnarray}
}
Substituting \eqref{pa6}, \eqref{pa5}, \eqref{pa4} and \eqref{pa3} in \eqref{pa2rev} yields
\begin{eqnarray}
\label{pa7}
\Vert II_4\Vert^2_{L^2(\Omega,H)}\leq C\Delta t^{\min(\beta,1-\epsilon)}+C\Delta t\sum_{i=0}^{m-1}\left\Vert X^h(t_{i})-X^h_{i}\right\Vert^2_{L^2(\Omega,H)}.
\end{eqnarray}
Substituting \eqref{pa7}, \eqref{cont5}, \eqref{isen6},  \eqref{mar6} and \eqref{deb1} in \eqref{num3} yields 
\begin{eqnarray}
\label{pa8}
\Vert X^h(t_m)-X^h_m\Vert^2_{L^2(\Omega,H)}\leq C\Delta t^{\min(\beta,1-\epsilon)}+C\Delta t\sum_{i=0}^{m-1}\left\Vert X^h(t_i)-X^h_i\right\Vert^2_{L^2(\Omega,H)}.
\end{eqnarray}
Applying the discrete Gronwall lemma to \eqref{pa8} yields
\begin{eqnarray}
\label{pa9}
\Vert X^h(t_m)-X^h_m\Vert_{L^2(\Omega,H)}\leq C\Delta t^{\frac{\min(\beta,1-\epsilon)}{2}}.
\end{eqnarray}
This completes the proof of \thmref{mainresult1} (i)-(ii).  Note that to prove \thmref{mainresult1} (iii) we only need 
to re-estimate $\Vert II_{43}\Vert^2_{L^2(\Omega,H)}$ by using \assref{assumption5} to achieve optimal convergence order $\frac{1}{2}$.

\subsection{Proof of \thmref{mainresult2}}
Let us recall that
\begin{eqnarray}
\label{fin0}
\Vert X^h(t_m)-X^h_m\Vert^2_{L^2(\Omega,H)}\leq 25\sum_{i=2}^{4}\Vert III_i\Vert_{L^2(\Omega,H)},
\end{eqnarray}
where $III_0$, $III_1$ and $III_3$ are exactly the same as $II_0$, $II_1$  and $II_3$ respectively. Therefore  \eqref{mar6} and \eqref{deb1} yields
\begin{eqnarray}
\label{fin1}
\Vert III_0\Vert_{L^2(\Omega,H)}+\Vert III_1\Vert_{L^2(\Omega,H)}\leq C\Delta t^{\frac{\beta}{2}}+C\Delta t\Vert X^h(t_{m-1})-X^h_{m-1}\Vert_{L^2(\Omega,H)}. 
\end{eqnarray}
It remains to estimate $III_3$ and the terms involving the noise, which are given below 
\begin{eqnarray}
III_2&=&\int_{t_{m-1}}^{t_m}\left[U_h(t_m,s)-S^{m-1}_{h,\Delta t}\right]P_hdW(s),\\
III_4&=&\sum_{i=2}^{m}\int_{t_{m-i}}^{t_{m-i+1}}\left(\prod_{j=m-i+1}^{m}U_h(t_j, t_{j-1})\right)U_h(t_{m-i+1}, s)P_hdW(s)\nonumber\\
&&-\sum_{i=2}^{m}\int_{t_{m-i}}^{t_{m-i+1}}\left(\prod_{j=m-i}^{m-1}S^j_{h,\Delta t}\right)P_hdW(s).
\end{eqnarray}
\subsubsection{Estimate of $III_2$}
We can split $III_2$ in two terms as follows:
\begin{eqnarray}
\label{dor1}
III_2&=&\int_{t_{m-1}}^{t_m}\left[U_h(t_m,s)-U_h(t_m, t_{m-1})\right]P_hdW(s)+\int_{t_{m-1}}^{t_m}\left[U_h(t_m, t_{m-1})-S^{m-1}_{h,\Delta t}\right]P_hdW(s)\nonumber\\
&:=&III_{21}+III_{22}.
\end{eqnarray}
Using the it\^{o} isometry, Lemmas \ref{evolutionlemma} and \ref{lemma11} (i), it holds that
\begin{eqnarray}
\label{dor2}
&&\Vert III_{21}\Vert^2_{L^2(\Omega,H)}\nonumber\\
&=&\int_{t_{m-1}}^{t_m}\left\Vert\left[U_h(t_m, s)-U_h(t_m, t_{m-1})\right]P_hQ^{\frac{1}{2}}\right\Vert^2_{\mathcal{L}_2(H)}ds\nonumber\\
&\leq& \int_{t_{m-1}}^{t_m}\left\Vert U_h(t_m,s)\left(\mathbf{I}-U_h(s, t_{m-1})\right)(A_{h,m-1})^{\frac{1-\beta}{2}}\right\Vert^2_{L(H)}\left\Vert(A_{h,m-1})^{\frac{\beta-1}{2}}P_hQ^{\frac{1}{2}}\right\Vert^2_{\mathcal{L}_2(H)}ds\nonumber\\
&\leq& \int_{t_{m-1}}^{t_m}\left\Vert U_h(t_m, s)(A_{h,m-1})^{\frac{1-\epsilon}{2}}\right\Vert^2_{L(H)}\left\Vert(A_{h, m-1})^{\frac{-1+\epsilon}{2}}\left(\mathbf{I}-U_h(s, t_{m-1})\right)(A_{h,m-1})^{\frac{1-\beta}{2}}\right\Vert^2_{L(H)}\nonumber\\
&&\times\left\Vert(A_{h,m-1})^{\frac{\beta-1}{2}}P_hQ^{\frac{1}{2}}\right\Vert^2_{\mathcal{L}_2(H)}ds\nonumber\\
&\leq& C\int_{t_{m-1}}^{t_m}(t_m-s)^{-1+\epsilon}(s-t_{m-1})^{\beta-\epsilon}ds\leq C\Delta t^{\beta-\epsilon}\int_{t_{m-1}}^{t_m}(t_m-s)^{-1+\epsilon}ds\leq C\Delta t^{\beta}.
\end{eqnarray}
Applying again the It\^{o} isometry, using \lemref{reviewimportant}  and \lemref{lemma11} (i) yields
\begin{eqnarray}
\label{dor3}
\Vert III_{22}\Vert^2_{L^2(\Omega,H)}&=&\int_{t_{m-1}}^{t_m}\left\Vert\left[U_h(t_m, t_{m-1})-S^{m-1}_{h,\Delta t}\right]P_hQ^{\frac{1}{2}}\right\Vert^2_{\mathcal{L}_2(H)}ds\nonumber\\
&\leq& C\int_{t_{m-1}}^{t_m}\Delta t^{\beta-1}\left\Vert (A_{h,m-1})^{\frac{\beta-1}{2}}P_hQ^{\frac{1}{2}}\right\Vert^2_{\mathcal{L}_2(H)}ds\nonumber\\
&\leq &C\Delta t^{\beta}.
\end{eqnarray}
Substituting \eqref{dor3}, \eqref{dor2} in \eqref{dor1} yields
\begin{eqnarray}
\label{fin2}
\Vert III_2\Vert^2_{L^2(\Omega, H)}\leq 2\Vert III_{21}\Vert^2_{L^2(\Omega,H)}+2\Vert III_{22}\Vert^2_{L^2(\Omega,H)}\leq C\Delta t^{\beta}.
\end{eqnarray}

\subsubsection{Estimate of $III_3$} 
Since $III_3$ is the same as $II_3$, it follows from \eqref{cont0} that
\begin{eqnarray}
\label{nlast0}
III_3=III_{31}+III_{32}+III_{33}+III_{34},
\end{eqnarray}
where $III_{31}$ $III_{32}$, $III_{33}$ and $III_{34}$ are respectively $II_{31}$ $II_{32}$, $II_{33}$ and $II_{34}$. Therefore from \eqref{cont1}, \eqref{cont3} and \eqref{cont4} we have 
\begin{eqnarray}
\label{nlast1}
&&\Vert III_{31}\Vert_{L^2(\Omega,H)}+\Vert III_{33}\Vert_{L^2(\Omega,H)}+\Vert III_{34}\Vert_{L^2(\Omega,H)}\nonumber\\
&\leq& C\Delta t^{\beta}+C\Delta t\sum_{i=2}^{m}\Vert X^h(t_{m-i})-X^h_{m-i}\Vert_{L^2(\Omega,H)}.
\end{eqnarray}
To achieve higher order we need to re-estimate $III_{32}$ by using the additional \assref{assumption7}. Note that $III_{32}$ can be recast as follows:
\begin{eqnarray}
\label{pa0}
&&III_{32}\nonumber\\
&=&\sum_{i=2}^{m}\int_{t_{m-i}}^{t_{m-i+1}}\left(\prod_{j=m-i+1}^{m}U_h(t_j, t_{j-1})\right)\left[P_hF\left(s,X^h(s)\right)-P_hF\left(t_{m-i+1}, X^h(t_{m-i})\right)\right]ds\nonumber\\
&=&\sum_{i=2}^{m}\int_{t_{m-i}}^{t_{m-i+1}}\left(\prod_{j=m-i+1}^{m}U_h(t_j, t_{j-1})\right)\left[P_hF\left(s,X^h(s)\right)-P_hF\left(t_{m-i+1}, X^h(s)\right)\right]ds\nonumber\\
&+&\sum_{i=2}^{m}\int_{t_{m-i}}^{t_{m-i+1}}\left(\prod_{j=m-i+1}^{m}U_h(t_j, t_{j-1})\right)\left[P_hF\left(t_{m-i+1},X^h(s)\right)-P_hF\left(t_{m-i+1}, X^h(t_{m-i})\right)\right]ds\nonumber\\
&:=&III_{321}+III_{322}.
\end{eqnarray}
Using triangle inequality, Lemmas \ref{pazylemma} (iii), \ref{evolutionlemma} and \assref{assumption3}, it holds that
\begin{eqnarray}
\label{pa1}
\Vert III_{321}\Vert_{L^2(\Omega,H)}\leq C\sum_{i=2}^{m-1}\int_{t_{m-i}}^{t_{m-i+1}}\Vert U_h(t_m, t_{m-i})\Vert_{L(H)}(t_{m-i+1}-s)^{\frac{\beta}{2}}ds\leq C\Delta t^{\frac{\beta}{2}}.
\end{eqnarray}
 For the seek of ease of notations, we set
\begin{eqnarray}
G^h_{m-i}\left(u\right):=P_hF\left(t_{m-i+1},u\right),\quad u\in H,\quad i=2,\cdots,m,\quad m=2,\cdots,M.
\end{eqnarray}
Applying  Taylor's formula in Banach space as in \cite{Arnulf3} yields
\begin{eqnarray}
\label{refait1}
G^h_{m-i}(X^h(s))-G^h_{m-i}(X^h(t_{m-i}))=I^h_{m,i}(s)\left(X^h(s)-X^h(t_{m-i})\right),
\end{eqnarray}
where $I^h_{m,i}(s)$ is defined for $t_{m-i}\leq s\leq t_{m-i+1}$ as follows:
\begin{eqnarray}
\label{refait2}
I^h_{m,i}(s):=\int_0^1\left(G^h_{m-i}\right)'\left(X^h(t_{m-i})+\lambda\left(X^h(s)-X^h(t_{m-i})\right)\right)d\lambda.
\end{eqnarray}
Using \assref{assumption7} and \lemref{lemma11} (ii), one can easily check that
\begin{eqnarray}
\label{refait2a}
\Vert I^h_{m,i}(s)\Vert_{L(H)}\leq C,\quad m\in\{1, \cdots,M\},\; 0\leq i\leq m-1,\; t_{m-i}\leq s\leq t_{m-i+1}.
\end{eqnarray}
Note that the mild solution $X^h(s)$ (with $t_{m-i}\leq s\leq t_{m-i+1}$, $i=2,\cdots, m$) can be written as follows
\begin{eqnarray}
\label{refait3}
X^h(s)=U_h(s, t_{m-i})X^h(t_{m-i})+\int_{t_{m-i}}^sU_h(s,r)P_hF(r,X^h(r))dr
+\int_{t_{m-i}}^sU_h(s,r)P_hdW(r).
\end{eqnarray}
Substituting \eqref{refait3} in \eqref{refait1} yields
\begin{eqnarray}
\label{refait4}
&&G^h_{m-i}(X^h(s))-G^h_{m-i}(X^h(t_{m-i}))\nonumber\\
&=&I^h_{m,i}(s)\left(U_h(s, t_{m-i})-\mathbf{I}\right)X^h(t_{m-i})+I^h_{m,i}(s)\int_{t_{m-i}}^sU_h(s,r)P_hF\left(X^h(r)\right)dr\nonumber\\
&+&I^h_{m,i}(s)\int_{t_{m-i}}^sU_h(s,r)P_hdW(r),\quad t_{m-i}\leq s\leq t_{m-i+1}.
\end{eqnarray}
Substituting \eqref{refait4} in the expression of $III_{322}$ (see \eqref{pa0}) yields
\begin{eqnarray}
\label{cle0}
III_{322}
&=&\sum_{i=2}^{m}\int_{t_{m-i}}^{t_{m-i+1}}\left(\prod_{j=m-i+1}^{m}U_h(t_j, t_{j-1})\right)I^h_{m,i}(s)\left(U_h(s, t_{m-i})-\mathbf{I}\right)X^h(t_{m-i})ds\nonumber\\
&+&\sum_{i=2}^{m}\int_{t_{m-i}}^{t_{m-i+1}}\left(\prod_{j=m-i+1}^{m}U_h(t_j, t_{j-1})\right)I^h_{m,i}(s)\int_{t_{m-i}}^sU_h(s, r)P_hF(r, X^h(r))drds\nonumber\\
&+&\sum_{i=2}^{m}\int_{t_{m-i}}^{t_{m-i+1}}\left(\prod_{j=m-i+1}^{m}U_h(t_j, t_{j-1})\right)I^h_{m,i}(s)\int_{t_{m-i}}^sU_h(s,r)P_hdW(r)ds\nonumber\\
&:=&III_{322}^{(1)}+III_{322}^{(2)}+III_{322}^{(3)}.
\end{eqnarray}
Inserting an appropriate power of $A_{h,m-i}$, using  Lemmas \ref{pazylemma} (iii), \ref{evolutionlemma}, \eqref{refait2a} and \lemref{lemma3}, it holds that
\begin{eqnarray}
\label{cle1}
\Vert III_{322}^{(1)}\Vert_{L^2(\Omega,H)}
&\leq&\sum_{i=2}^{m}\int_{t_{m-i}}^{t_{m-i+1}}\left\Vert\left(\prod_{j=m-i+1}^{m}U_h(t_j, t_{j-1})\right)\right\Vert_{L(H)}\Vert I^h_{m,i}(s)\Vert_{L(H)}\nonumber\\
&&\times\left\Vert\left(U_h(s,t_{m-i})-\mathbf{I}\right)(A_{h,m-i})^{-\frac{\beta}{2}+\epsilon}\right\Vert_{L(H)}\Vert (A_{h,m-i})^{\frac{\beta}{2}-\epsilon} X^h(t_{m-i})\Vert_{L^2(\Omega,H)}ds\nonumber\\
&\leq& C\sum_{i=2}^{m}\int_{t_{m-i}}^{t_{m-i+1}}\Vert U_h(t_m, t_{m-i})\Vert_{L(H)}(s-t_{m-i})^{\frac{\beta}{2}-\epsilon}ds\nonumber\\
&\leq& C\Delta t^{\frac{\beta}{2}-\epsilon}\sum_{i=2}^{m}t_i^{-1+\epsilon}\Delta t\leq C\Delta t^{\frac{\beta}{2}-\epsilon}.
\end{eqnarray}
Using Lemmas \ref{pazylemma} (iii), \ref{evolutionlemma}, \ref{lemma11} (ii), \eqref{refait2a}, \coref{corollary1}  and \eqref{smooth1} yields
\begin{eqnarray}
\label{cle2}
\Vert III_{322}^{(2)}\Vert_{L^2(\Omega,H)}
&\leq& C\sum_{i=2}^{m}\int_{t_{m-i}}^{t_{m-i+1}}\left\Vert \int_{t_{m-i}}^sU_h(s, r)P_hF(r, X^h(r))dr\right\Vert_{L^2(\Omega,H)}ds\nonumber\\
&\leq& C\sum_{i=2}^{m}\int_{t_{m-i}}^{t_{m-i+1}}(s-t_{m-i})ds\leq C\Delta t.
\end{eqnarray}
We split $III_{322}^{(3)}$ in two terms as follows
\begin{eqnarray}
\label{Raphael1}
III_{322}^{(3)}&=&\sum_{i=2}^{m}\int_{t_{m-i}}^{t_{m-i+1}}\left(\prod_{j=m-i+1}^{m}U_h(t_j, t_{j-1})\right)I^h_{m,i}(t_{m-i})\int_{t_{m-i}}^sU_h(s,r)P_hdW(r)ds\nonumber\\
&+&\sum_{i=2}^{m}\int_{t_{m-i}}^{t_{m-i+1}}\left(\prod_{j=m-i+1}^{m}U_h(t_j, t_{j-1})\right)\left(I^h_{m,i}(s)-I^h_{m,i}(t_{m-i})\right)\int_{t_{m-i}}^sU_h(s,r)P_hdW(r)ds\nonumber\\
&=:&III_{322}^{(31)}+III_{322}^{(32)}.
\end{eqnarray}
Since the expression in $III_{322}^{(31)}$ is $\mathcal{F}_{t_{m-i}}$-measurable,  the expectation of the cross-product vanishes. Using the It\^{o} isometry, triangle inequality, H\"{o}lder inequality, Lemmas \ref{pazylemma} (iii) and \ref{evolutionlemma} yields 
\begin{eqnarray}
\label{petit0}
&&\Vert III_{322}^{(31)}\Vert^2_{L^2(\Omega,H)}\nonumber\\
&=&\mathbb{E}\left[\left\Vert\sum_{i=2}^{m}\int_{t_{m-i}}^{t_{m-i+1}}\left(\prod_{j=m-i+1}^{m}U_h(t_j, t_{j-1})\right)I^h_{m,i}(t_{m-i})\int_{t_{m-i}}^sU_h(s, r)P_hdW(r) ds\right\Vert^2\right]\nonumber\\
&=&\sum_{i=2}^{m}\mathbb{E}\left[\left\Vert\int_{t_{m-i}}^{t_{m-i+1}}\int_{t_{m-i}}^s\left(\prod_{j=m-i+1}^{m}U_h(t_j, t_{j-1})\right)I^h_{m,i}(t_{m-i})U_h(s, r)P_hdW(r) ds\right\Vert^2\right]\nonumber\\
&\leq&\Delta t\sum_{i=2}^{m}\int_{t_{m-i}}^{t_{m-i+1}}\mathbb{E}\left[\left\Vert\int_{t_{m-i}}^s\left(\prod_{j=m-i+1}^{m}U_h(t_j, t_{j-1})\right)I^h_{m,i}(t_{m-i})U_h(s,r)P_hdW(r) \right\Vert^2\right]ds\nonumber\\
&\leq&C\Delta t\sum_{i=2}^{m}\int_{t_{m-i}}^{t_{m-i+1}}\int_{t_{m-i}}^s\mathbb{E}\left\Vert\left(\prod_{j=m-i+1}^{m}U_h(t_j, t_{j-1})\right)I^h_{m,i}(t_{m-i})U_h(s, r)P_hQ^{\frac{1}{2}}\right\Vert^2_{\mathcal{L}_2(H)} dr ds\nonumber\\
&\leq&C\Delta t\sum_{i=2}^{m}\int_{t_{m-i}}^{t_{m-i+1}}\int_{t_{m-i}}^s\mathbb{E}\left\Vert I^h_{m,i}(t_{m-i})U_h(s,\sigma)P_hQ^{\frac{1}{2}}\right\Vert^2_{\mathcal{L}_2(H)} dr ds.
\end{eqnarray}
Using Lemmas \ref{evolutionlemma}, \ref{lemma11} (ii) and \eqref{refait2a}  yields
\begin{eqnarray}
\label{petit1}
&&\mathbb{E}\left\Vert I^h_{m,i}(t_{m-i})U_h(s,r)P_hQ^{\frac{1}{2}}\right\Vert^2_{\mathcal{L}_2(H)}\nonumber\\
&=&\mathbb{E}\left\Vert I^h_{m,i}(s)U_h(s,r)(A_{h,m-i-1})^{\frac{1-\beta}{2}}(A_{h,m-i-1})^{\frac{\beta-1}{2}} P_hQ^{\frac{1}{2}}\right\Vert^2_{\mathcal{L}_2(H)}\nonumber\\
&\leq&\mathbb{E}\left\Vert I^h_{m,i}(s)U_h(s,r)(A_{h,m-i-1})^{\frac{1-\beta}{2}}\right\Vert^2_{L(H)}\left\Vert(A_{h,m-i-1})^{\frac{\beta-1}{2}} P_hQ^{\frac{1}{2}}\right\Vert^2_{\mathcal{L}_2(H)}\nonumber\\
&\leq&\mathbb{E}\left\Vert U_h(s,r)(A_{h,m-i-1})^{\frac{1-\beta}{2}}\right\Vert^2_{L(H)}\left\Vert(A_{h,m-i-1})^{\frac{\beta-1}{2}} P_hQ^{\frac{1}{2}}\right\Vert^2_{\mathcal{L}_2(H)}\nonumber\\
&\leq & C(s-r)^{\min(-1+\beta,0)}.
\end{eqnarray}
Substituting \eqref{petit1} in \eqref{petit0} yields 
\begin{eqnarray}
\label{cle3}
\Vert III_{322}^{(31)}\Vert^2_{L^2(\Omega,H)}\leq C\Delta t\sum_{i=2}^{m}\int_{t_{m-i}}^{t_{m-i+1}}\int_{t_{m-i}}^s(s-r)^{\min(-1+\beta,0)}dr ds\leq C\Delta t^{\min(1+\beta,2)}.
\end{eqnarray}
Using triangle inequality, Cauchy-Schwartz inequality and  the fact that $U_h(t,s)U_h(s,r)=U_h(t,r)$, $0\leq r\leq s\leq t$, yields
 \begin{eqnarray}
 \label{mach5a}
 &&\Vert III_{322}^{(32)}\Vert^2_{L^2(\Omega, H)}\\
 &\leq& m\sum_{i=2}^{m}\left\Vert\int_{t_{m-i}}^{t_{m-i+1}}U_h(t_m,t_{m-i})P_h\left(I^h_{m,i}(s)-I^h_{m,i}(t_{m-i})\right)\int_{t_{m-i}}^sU_h(s,r)P_hdW(r)ds\right\Vert^2_{L^2(\Omega, H)}\nonumber\\
 &\leq& m\Delta t\sum_{i=2}^{m}\int_{t_{m-i}}^{t_{m-i+1}}\left\Vert U_h(t_m,t_{m-i})P_h\left(I^h_{m,i}(s)-I^h_{m,i}(t_{m-i})\right)\int_{t_{m-i}}^s U_h(s,r)P_hdW(r)\right\Vert^2_{L^2(\Omega, H)}ds\nonumber\\
 &\leq& T\sum_{i=2}^{m}\int_{t_{m-i}}^{t_{m-i+1}}\left(\mathbb{E}\left\Vert U_h(t_m,t_{m-i})P_h\left(I^h_{m,i}(s)-I^h_{m,i}(t_{m-i})\right)\right\Vert^4_{L(H)}\right)^{\frac{1}{2}}\nonumber\\
 &\times&\left(\mathbb{E}\left\Vert\int_{t_{m-i}}^s U_h(s,r)P_hdW(r)\right\Vert^4\right)^{\frac{1}{2}} ds.
 \end{eqnarray}
 Using the Burkh\"{o}lder-Davis-Gundy inequality (\cite[Lemma 5.1]{Kruse1}), it follows that
 \begin{eqnarray}
 \label{mach5}
 \Vert II_{322}^{(32)}\Vert^2_{L^2(\Omega, H)}
 &\leq&T\sum_{i=2}^{m}\int_{t_{m-i}}^{t_{m-i+1}}\left(\mathbb{E}\left\Vert U_h(t_m,t_{m-i})P_h\left(I^h_{m,i}(s)-I^h_{m,i}(t_{m-i})\right)\right\Vert^4_{L(H)}\right)^{\frac{1}{2}}\nonumber\\
 &\times&\left(\int_{t_{m-i}}^s\left\Vert  U_h(s,r)P_hQ^{\frac{1}{2}}\right\Vert^2_{\mathcal{L}_2(H)}dr\right) ds.
 \end{eqnarray}
Using Lemmas \ref{lemma11} and \ref{evolutionlemma}, it holds that
 \begin{eqnarray}
 \label{mach6a}
 \left\Vert U_h(s,r)P_hQ^{\frac{1}{2}}\right\Vert^2_{\mathcal{L}_2(H)}
 &\leq&\left\Vert U_h(s,r)(-A_h(r))^{\frac{1-\beta}{2}}(-A_h(r))^{\frac{\beta-1}{2}} P_hQ^{\frac{1}{2}}\right\Vert^2_{\mathcal{L}_2(H)}\nonumber\\
 &\leq&\left\Vert U_h(s,r)(-A_h(r))^{\frac{1-\beta}{2}}\right\Vert^2_{L(H)}\left\Vert (-A_h(r))^{\frac{\beta-1}{2}} P_hQ^{\frac{1}{2}}\right\Vert^2_{\mathcal{L}_2(H)}\nonumber\\
 &\leq& C(s-r)^{\min(0, \beta-1)}.
 \end{eqnarray} 
Using  \lemref{evolutionlemma}, it holds that
 \begin{eqnarray}
 \label{mach6b}
 &&\mathbb{E}\left\Vert U_h(t_m,t_{m-i})P_h\left(I^h_{m,i}(s)-I^h_{m,i}(t_{m-i})\right)\right\Vert_{L(H)}\nonumber\\
 &\leq&\left\Vert U_h(t_m,t_{m-i})(-A_h(s))^{\frac{\eta}{2}}\right\Vert_{L(H)}\mathbb{E}\left\Vert(-A_h(s))^{-\frac{\eta}{2}}\left(I^h_{m,i}(s)-I^h_{m,i}(t_{m-i})\right)\right\Vert_{L( H)}\nonumber\\
 &\leq& Ct_i^{-\frac{\eta}{2}}\mathbb{E}\left\Vert (-A_h(s))^{-\frac{\eta}{2}}\left(I^h_{m,i}(s)-I^h_{m,i}(t_{m-i})\right)\right\Vert_{L(H)}.
 \end{eqnarray}
 From the definition of $I^h_{m,i}(s)$ \eqref{refait2}, by using  \lemref{lemma11}  we arrive at 
 \begin{eqnarray}
 \label{mach7}
 &&\Vert (-A_h(s))^{-\frac{\eta}{2}}\left(I^h_{m,i}(s)-I^h_{m,i}(t_{m-i}\right)\Vert_{\mathcal{L}(H)}\nonumber\\
&\leq& \int_0^1\left\Vert (-A_h(s))^{-\frac{\eta}{2}}P_h\left( F'\left(t_{m-i},X^h(t_{m-i})+\lambda\left(X^h(s)-X^h(t_{m-i})\right)\right)\right.\right.\nonumber\\
&&\left.\left.-F'\left(t_{m-i}, X^h(t_{m-i})\right)\right)
\right\Vert_{\mathcal{L}(H)}d\lambda\nonumber\\
&\leq& \int_0^1\left\Vert (-A(s))^{-\frac{\eta}{2}}\left( F'\left(t_{m-i},X^h(t_{m-i})+\lambda\left(X^h(s)-X^h(t_{m-i})\right)\right)\right.\right.\nonumber\\
&&\left.\left.-F'\left(t_{m-k-1}, X^h(t_{m-i})\right)\right)\right\Vert_{\mathcal{L}(H)}d\lambda\nonumber\\
&\leq& C\int_0^1\lambda\Vert X^h(s)-X^h(t_{m-i})\Vert d\lambda\nonumber\\
&\leq & C\Vert X^h(s)-X^h(t_{m-i})\Vert.
 \end{eqnarray}
 Substituting \eqref{mach7} in \eqref{mach6b} and   using \lemref{lemma3} yields
 \begin{eqnarray}
 \label{mach8}
 \mathbb{E}\left\Vert U_h(t_m,t_{m-i})P_h\left(I^h_{m,i}(s)-I^h_{m,i}(t_{m-i}\right)\right\Vert^4_{L(H)}
 &\leq& Ct_i^{-2\eta}\mathbb{E}\Vert X^h(s)-X^h(t_{m-i})\Vert^4\nonumber\\
 &\leq& Ct_i^{-2\eta}(s-t_{m-i})^{\min(2, 2\beta)}.
 \end{eqnarray}
 Substituting \eqref{mach8} and \eqref{mach6a} in \eqref{mach5} yields
 \begin{eqnarray}
 \label{mach9}
 \Vert II_{322}^{(32)}\Vert^2_{L^2(\Omega, H)}
 &\leq& C\sum_{i=2}^{m}\int_{t_{m-i}}^{t_{m-i+1}}\int_{t_{m-i}}^st_i^{-\eta}(s-r)^{\min(0, \beta-1)}(s-t_{m-i})^{\min(1, \beta)}drds\nonumber\\
 &\leq&  C\sum_{i=2}^{m}\int_{t_{m-i}}^{t_{m-i+1}}t_i^{-\eta}(s-t_{m-i})^{\min(2, 2\beta)}ds\nonumber\\
 &\leq& C\Delta t^{\min(2, 2\beta)}\sum_{i=2}^{m}\int_{t_{m-i}}^{t_{m-i+1}}t_i^{-\eta}ds\leq C\Delta t^{\min(2, 2\beta)}.
 \end{eqnarray}
 Substituting \eqref{mach9} and \eqref{cle3} in \eqref{Raphael1} yields
 \begin{eqnarray}
 \label{mach10}
 \Vert II_{322}^{(3)}\Vert_{L^2(\Omega, H)}\leq C\Delta t^{\frac{\beta}{2}}.
 \end{eqnarray}
Substituting  \eqref{mach10}, \eqref{cle2} and \eqref{cle1} in \eqref{cle0} yields
\begin{eqnarray}
\label{pa2}
\Vert III_{322}\Vert_{L^2(\Omega,H)}
&\leq& \Vert III_{322}^{(1)}\Vert_{L^2(\Omega,H)}+\Vert III_{322}^{(2)}\Vert_{L^2(\Omega,H)}+\Vert III_{322}^{(3)}\Vert_{L^2(\Omega,H)}+\Vert III_{322}^{(4)}\Vert_{L^2(\Omega,H)}\nonumber\\
&\leq &C\Delta t^{\frac{\beta}{2}-\epsilon}.
\end{eqnarray}
Substituting \eqref{pa2} and \eqref{pa1} in \eqref{pa0} yields
\begin{eqnarray}
\label{nlast2}
\Vert III_{32}\Vert_{L^2(\Omega,H)}\leq \Vert III_{321}\Vert_{L^2(\Omega,H)}+\Vert III_{322}\Vert_{L^2(\Omega,H)}\leq C\Delta t^{\frac{\beta}{2}-\epsilon}.
\end{eqnarray}
Substituting \eqref{nlast2} and \eqref{nlast1} in \eqref{nlast0} yields 
\begin{eqnarray}
\label{fin8}
\Vert III_3\Vert_{L^2(\Omega,H)}\leq C\Delta t^{\frac{\beta}{2}-\epsilon}+C\Delta t\sum_{i=1}^{m-1}\Vert X^h(t_i)-X^h_i\Vert_{L^2(\Omega, H)}.
\end{eqnarray}
\subsubsection{Estimate of $III_4$}
We can recast $III_4$ in two terms as follows
\begin{eqnarray}
\label{ter1}
 III_4&=&\sum_{i=2}^{m}\int_{t_{m-i}}^{t_{m-i+1}}\left(\prod_{j=m-i+2}^{m}U_h(t_j, t_{j-1})\right)\left(U_h(t_{m-i+1}, s)-U_h(t_{m-i+1}, t_{m-i})\right)P_hdW(s)\nonumber\\
&+&\sum_{i=2}^{m}\int_{t_{m-i}}^{t_{m-i+1}}\left[\left(\prod_{j=m-i+1}^{m}U_h(t_j, t_{j-1})\right)-\left(\prod_{j=m-i}^{m-1}S^j_{h,\Delta t}\right)\right]P_hdW(s)\nonumber\\
&:=&III_{41}+III_{42}.
\end{eqnarray}
Using the It\^{o} isometry, Lemmas \ref{pazylemma} (iii), \ref{evolutionlemma} and \ref{lemma11} (i)  yields 
\begin{eqnarray}
\label{ter2}
&&\Vert III_{41}\Vert^2_{L^2(\Omega,H)}\nonumber\\
&=&\sum_{i=2}^{m}\int_{t_{m-i}}^{t_{m-i+1}}\left\Vert\left(\prod_{j=m-i+1}^{m}U_h(t_j, t_{j-1})\right)\left(U_h(t_{m-i+1}, s)-U_h(t_{m-i+1}, t_{m-i})\right)P_hQ^{\frac{1}{2}}\right\Vert^2_{\mathcal{L}_2(H)}ds\nonumber\\
&\leq&C\sum_{i=2}^{m}\int_{t_{m-i}}^{t_{m-i+1}}\left\Vert\left(\prod_{j=m-i+1}^{m-1}U_h(t_j, t_{j-1})\right)\left(\mathbf{I}-U_h(s,t_{m-i})\right)(A_{h,m-1})^{\frac{1-\beta}{2}}\right\Vert^2_{L(H)}\nonumber\\
&&\times\left\Vert (A_{h,m-1})^{\frac{\beta-1}{2}}U_h(t_{m-i+1},s)(A_{h,m-1})^{\frac{1-\beta}{2}}\right\Vert_{L(H)}\left\Vert (A_{h,m-1})^{\frac{\beta-1}{2}} P_hQ^{\frac{1}{2}}\right\Vert^2_{\mathcal{L}_2(H)}ds\nonumber\\
&\leq&C\sum_{i=2}^{m}\int_{t_{m-i}}^{t_{m-i+1}}\left\Vert\left(\prod_{j=m-i+1}^{m}U_h(t_j, t_{j-1})\right)(A_{h,m-1})^{\frac{1-\epsilon}{2}}\right\Vert^2_{L(H)}\nonumber\\
&&\times\left\Vert(A_{h,m-1})^{\frac{-1+\epsilon}{2}}\left(\mathbf{I}-U_h(s,t_{m-i})\right)(A_{h,m-1})^{\frac{1-\beta}{2}}\right\Vert^2_{L(H)}ds\nonumber\\
&\leq& C\sum_{i=2}^{m}\int_{t_{m-i}}^{t_{m-i+1}}\left\Vert U_h(t_m, t_{m-i})(A_{h,m-1})^{\frac{1-\epsilon}{2}}\right\Vert^2_{L(H)}(s-t_{m-i})^{\beta-\epsilon}ds\nonumber\\
&\leq& C\Delta t^{\beta-\epsilon}\sum_{i=2}^{m}t_i^{-1+\epsilon}\Delta t\leq C\Delta t^{\beta-\epsilon}.
\end{eqnarray}
Using again the It\^{o} isometry, \lemref{lemma11} (i) yields
\begin{eqnarray}
\label{ter3}
&&\Vert III_{42}\Vert^2_{L^2(\Omega,H)}\\
&=&\sum_{i=2}^{m}\int_{t_{m-i}}^{t_{m-i+1}}\left\Vert \left[\left(\prod_{j=m-i+1}^{m}U_h(t_j, t_{j-1})\right)-\left(\prod_{j=m-i}^{m-1}S^j_{h,\Delta t}\right)\right]P_hQ^{\frac{1}{2}}\right\Vert^2_{\mathcal{L}_2(H)}ds\nonumber\\
&\leq&\sum_{i=2}^{m}\int_{t_{m-i}}^{t_{m-i+1}}\left\Vert \left[\left(\prod_{j=m-i+1}^{m}U_h(t_j, t_{j-1})\right)-\left(\prod_{j=m-i}^{m-1}S^j_{h,\Delta t}\right)\right](A_{h,m-i})^{\frac{1-\beta}{2}}\right\Vert^2_{L(H)}\nonumber\\
&&\times\left\Vert(A_{h,m-i})^{\frac{\beta-1}{2}} P_hQ^{\frac{1}{2}}\right\Vert^2_{\mathcal{L}_2(H)}ds\nonumber\\
&\leq&C\sum_{i=2}^{m}\int_{t_{m-i}}^{t_{m-i+1}}\left\Vert \left[\left(\prod_{j=m-i+1}^{m}U_h(t_j, t_{j-1})\right)-\left(\prod_{j=m-i}^{m-1}S^j_{h,\Delta t}\right)\right](A_{h,m-i})^{\frac{1-\beta}{2}}\right\Vert^2_{L(H)}ds.\nonumber
\end{eqnarray}
If $\beta<1$, then applying \lemref{reviewimportant} (iii) yields
\begin{eqnarray}
\label{ter4}
\Vert III_{42}\Vert^2_{L^2(\Omega,H)}
\leq C\sum_{i=2}^{m}\int_{t_{m-i}}^{t_{m-i+1}}\Delta t^{\beta-\epsilon} t_{i-1}^{-1+\epsilon}ds\leq C\Delta t^{\beta-\epsilon}\sum_{i=2}^{m}t_{i-1}^{1-\epsilon}\Delta t\leq C\Delta t^{\beta-\epsilon}.
\end{eqnarray}
If $\beta\geq 1$, then employing \lemref{reviewimportant} (ii), it follows from \eqref{ter3} that
\begin{eqnarray}
\label{ter5}
\Vert III_{42}\Vert^2_{L^2(\Omega,H)}\leq C\sum_{i=2}^{m}\int_{t_{m-i}}^{t_{m-i+1}}\Delta t^{\beta-\epsilon}t_{i-1}^{-1+\epsilon}ds\leq C\Delta t^{\beta-\epsilon}.
\end{eqnarray}
Therefore for all $\beta\in[0,2]$, it holds that
\begin{eqnarray}
\label{ter6}
\Vert III_{42}\Vert^2_{L^2(\Omega,H)}\leq C\Delta t^{\beta-\epsilon}.
\end{eqnarray}
Substituting \eqref{ter6} and \eqref{ter2} in \eqref{ter1} yields 
\begin{eqnarray}
\label{fin9}
\Vert III_4\Vert^2_{L^2(\Omega,H)}\leq 2\Vert III_{41}\Vert^2_{L^2(\Omega,H)}+2\Vert III_{42}\Vert^2_{L^2(\Omega,H)}\leq C\Delta t^{\beta-\epsilon}.
\end{eqnarray}
Substituting \eqref{fin9}, \eqref{fin8}, \eqref{fin2} and \eqref{fin1} in \eqref{fin0} yields 
\begin{eqnarray}
\label{fin10}
\Vert X^h(t_m)-X^h_m\Vert^2_{L^2(\Omega,H)}\leq C\Delta t^{\beta-\epsilon}+C\Delta t\sum_{i=0}^{m-1}\Vert X^h(t_{i})-X^h_{i}\Vert^2_{L^2(\Omega,H)}.
\end{eqnarray}
Applying the discrete Gronwall lemma to \eqref{fin10} yields
\begin{eqnarray*}
\Vert X^h(t_m)-X^h_m\Vert_{L^2(\Omega,H)}\leq C\Delta t^{\frac{\beta}{2}-\epsilon}.
\end{eqnarray*}
This completes the proof of \thmref{mainresult2}.

\section{Numerical experiments}
\label{numexperiment}
\subsection{Additive noise}
We consider the stochastic reaction diffusion equation 
\begin{eqnarray}
\label{linear}
  dX=[D(t) \varDelta X -k(t) X]dt+ dW, 
\qquad \text{given } \quad X(0)=X_{0},
\end{eqnarray}
in the time interval $[0,T]$ with diffusion coefficient $ D(t)=(1/10)(1+e^{-t})$ and reaction rate $k(t)=1$
on homogeneous Neumann boundary conditions on the domain
$\Lambda=[0,L_{1}]\times [0,L_{2}]$. We take $L_1=L_2=1$.
Our function $F(t,u)=k(t)u $ is linear and obviously satisfies \assref{assumption3}. 
Since $F(t, u)$ is linear on the second variable, it holds that $F'(t,u)v=k(t)v$ for all $u, v\in L^2(\Lambda)$, where $F'$ stands
for the differential  with respect to the second variable. 
Therefore $\Vert F'(t,u)\Vert_{L(H)}\leq\vert k(t)\vert =1$ for all $u\in L^2(\Lambda)$, hence \assref{assumption7} is fulfilled.
In general we are interested in nonlinear $F$. However, for this linear
system we can find a good approximation of the  exact solution to compare our numerics to.
The eigenfunctions $\{e_{i}^{(1)}\otimes e_{j}^{(2)}\}_{i,j\geq 0}
$ of the operator $\varDelta$ here are given by 
\begin{eqnarray}
e_{0}^{(l)}=\sqrt{\dfrac{1}{L_{l}}},\;\;\;\lambda_{0}^{(l)}=0,\;\;\;
e_{i}^{(l)}=\sqrt{\dfrac{2}{L_{l}}}\cos(\lambda_{i}^{(l)}x),\;\;\;\lambda_{i}^{(l)}=\dfrac{i
  \,\pi }{L_{l}},
\end{eqnarray}
where $l \in \left\lbrace 1, 2 \right\rbrace$ and  $i=\{1, 2, 3, \cdots\}$
with the corresponding eigenvalues $ \{\lambda_{i,j}\}_{i,j\geq 0} $ given by 
$\lambda_{i,j}= (\lambda_{i}^{(1)})^{2}+ (\lambda_{j}^{(2)})^{2}.$
The linear operator  $A(t)=D(t)\varDelta$ has the same eigenfunctions as $\varDelta$, but  with eigenvalues $\{D(t)\lambda_{i,j}\}_{i,j\geq 0}$.
Clearly we have $\mathcal{D}(A(t))=\mathcal{D}(A(0))$ and $\mathcal{D}((A(t))^{\alpha})=\mathcal{D}((A(0))^{\alpha})$ for all $t\in[0,T]$ and $0\leq \alpha\leq 1$.
Since $D(t)$ is bounded below by $(1/10)(1+e^{-T})$, it follows that the ellipticity condition \eqref{ellip} holds and therefore as a consequence of 
the analysis in Section \ref{numerics}, it follows that $A(t)$ are uniformly sectorial.
Obviously \assref{assumption2} is  also fulfilled.
We also used
\begin{eqnarray}
\label{noise2}
 q_{i,j}=\left( i^{2}+j^{2}\right)^{-(\beta +\delta)}, \, \beta>0
\end{eqnarray} 
in the representation \eqref{noise} for some small $\delta>0$. Here
the noise and the linear operator are supposed to have the same
eigenfunctions. We obviously have 
\begin{eqnarray}
\underset{(i,j) \in \mathbb{N}^{2}}{\sum}\lambda_{i,j}^{\beta-1}q_{i,j}<  \pi^{2}\underset{(i,j) 
\in \mathbb{N}^{2}}{\sum} \left( i^{2}+j^{2}\right)^{-(1+\delta)} <\infty,
\end{eqnarray}
thus \assref{assumption6} 
is satisfied.  In  our simulations, we take
$\beta\in \{1,1.5,2\}$, with $\delta=0.001$. The  close form  of  the exact solution of \eqref{linear} is known. Indeed, using the representation of the noise in \eqref{noise},
the decomposition of \eqref{linear} in each eigenvector node yields
the following  Ornstein-Uhlenbeck process 
\begin{eqnarray}
\label{exact}
 dX_{i}=-(D(t) \lambda_{i}+k(t))X_{i}dt+ \sqrt{q_{i}}d\beta_{i}(t),\qquad i \in \mathbb{N}^{2}.
\end{eqnarray}
This is a Gaussian process with the mild solution
\begin{eqnarray}
X_{i}(t)= e^{- \int_{0}^{t}b_{i}(s)ds} \left[ X_{i}(0)+  \sqrt{q_{i}}\int_{0}^{t}e^{ \int_{0}^{s} b_{i}(y)dy} d \beta_{i}(s)\right],\; b_{i}(t)=D(t) \lambda_{i}+k(t).
\end{eqnarray}
Applying the It\^{o} isometry yields the  following variance of $X_{i}(t)$
\begin{eqnarray}
\label{exact2}
 \text{Var}(X_{i}(t))=q_{i}\, e^{- \int_{0}^{t}\,2\,b_{i}(s)ds}  \left( \int_{0}^{t}e^{\int_{0}^{s} \,2\,b_{i}(y)dy} ds \right).
\end{eqnarray}
During simulation, we compute the exact solution recurrently as 
\begin{eqnarray}
\label{exact1}
  X_{i}^{m+1}  = e^{- \int_{t_m}^{t_{m+1}}b_{i}(s)ds} X_{i}^m +\left(q_{i}\, e^{- \int_{t_m}^{t_{m+1}}\,2\,b_{i}(s)ds}  \left( \int_{t_m}^{t_{m+1}}e^{\int_{t_m}^{s} \,2\,b_{i}(y)dy} ds \right) \right)^{1/2}R_{i,m},
  \end{eqnarray}
where $R_{i,m}$ are independent, standard normally distributed random
variables with mean $0$ and variance $1$.
Note that  the integrals involved in \eqref{exact1} are computed exactly  for the first integral and accurately appoximated for the second  integral.
 \begin{figure}[!ht]
 \begin{center}
 \includegraphics[width=0.70\textwidth]{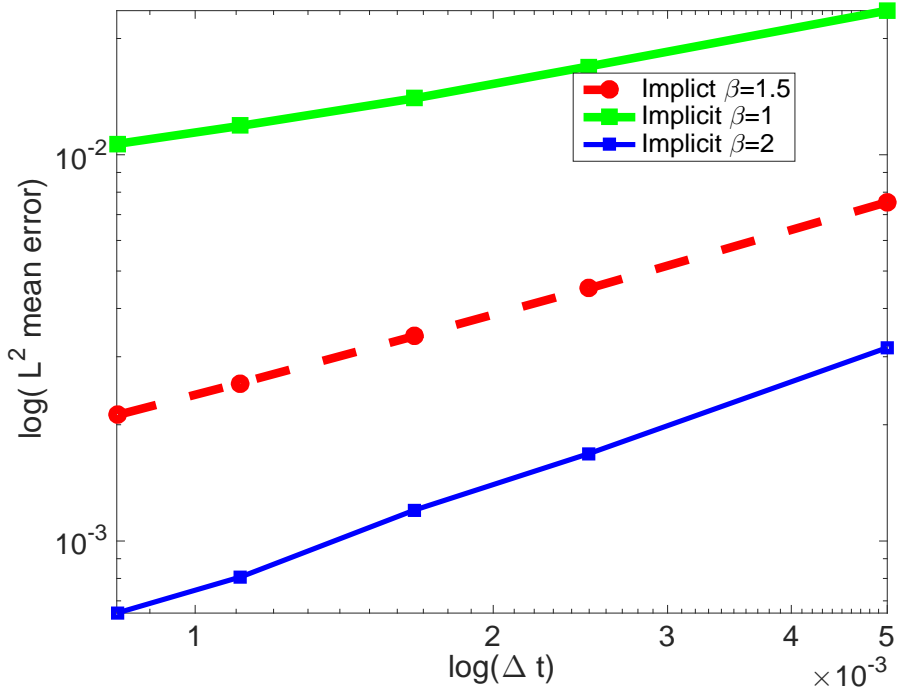}
  \end{center}
 \caption{Convergence of the  implicit scheme for $\beta=1$, $\beta=1.5$ and $\beta=2$ in \eqref{noise2} for SPDE \eqref{linear}.
 The orders of convergence in time  are $0.47$  for $\beta=1$, $0.72$ for $\beta=1.5$ and $0.93$ for $\beta=2$. The total number of samples used is $100$.}
 \label{FIGI}
 \end{figure}
 In \figref{FIGI}, we can observe the convergence of the the implicit scheme for three noise's parameters.
 Indeed, the order of convergence in time is $0.47$  for $\beta=1$, $0.721$ for $\beta=1.5$ and $0.93$ for $\beta=2$. 
 These orders are close to the theoretical  orders $0.5$, $0.75$ and $1$ obtained in \thmref{mainresult2} for $\beta=1$, $\beta=1.5$ and $\beta=2$ respectively.
\subsection{Multiplicative noise and application in porous media  flow}
We consider the following  stochastic  reactive dominated advection
diffusion equation  with  constant diagonal difussion tensor  
\begin{eqnarray}
\label{reactiondif1}
dX=\left[\left(1+e^{-t}\right) \left(\varDelta X-\nabla \cdot(\mathbf{q}X)\right)-\dfrac{e^{-t} X}{\vert X\vert +1}\right]dt+XdW,
\end{eqnarray}
with  mixed Neumann-Dirichlet boundary conditions on $\Lambda=[0,L_1]\times[0,L_2]$. 
The Dirichlet boundary condition is $X=1$ at $\Gamma=\{ (x,y) :\; x =0\}$ and 
we use the homogeneous Neumann boundary conditions elsewhere.
The eigenfunctions $ \{e_{i,j} \} =\{e_{i}^{(1)}\otimes e_{j}^{(2)}\}_{i,j\geq 0}
$ of  the covariance operator $Q$ are the same as for  Laplace operator $-\varDelta$  
with homogeneous boundary condition and we also use the noise representation \eqref{noise2}.
In our simulations, we take $\beta\in\{1.5, 2\}$ and $\delta=0.001$.
In \eqref{nemistekii1}, we take $b(x,u)=u$, $x\in \Lambda$ and $u\in \mathbb{R}$. Therefore, from \cite[Section 4]{Arnulf1} it follows that the operator  $B$ defined by \eqref{nemistekii1}
fulfills  Assumptions  \ref{assumption4} and \ref{assumption5}.  The function $F$ is given by $F(t,v)= -\dfrac{e^{-t} v}{1+\vert v\vert}$, $t\in[0, T]$, $v\in H$ and obviously satisfies \assref{assumption3}. The nonlinear operator $A(t)$ is given by 
\begin{eqnarray}
A(t)=(1+e^{-t})\left(\varDelta(.)-\nabla. \mathbf{q}(.)\right),\quad t\in[0, T],
\end{eqnarray}
where $\mathbf{q}$ is the Darcy velocity obtained as in \cite{Antonio1}.
Clearly $\mathcal{D}(A(t))=\mathcal{D}(A(0))$, $t\in[0, T]$ and $\mathcal{D}((A(t))^{\alpha})=\mathcal{D}((A(0))^{\alpha})$, $t\in[0, T]$, $0\leq \alpha\leq 1$. 
The  function $q_{i,j}(x,t)$ defined in \eqref{family} is given by $q_{i,j}(x,t)=1+e^{-t}$. Since $q_{i,j}(x,t)$ is bounded below by $1+e^{-T}$, it follows that the ellipticity  condition \eqref{ellip} holds and therefore as a consequence of \secref{numerics}, it follows that $A(t)$ is sectorial. Obviously \assref{assumption2} is fulfills. 
\begin{figure}[!ht]
 \begin{center}
 (a)
 \includegraphics[width=0.7\textwidth]{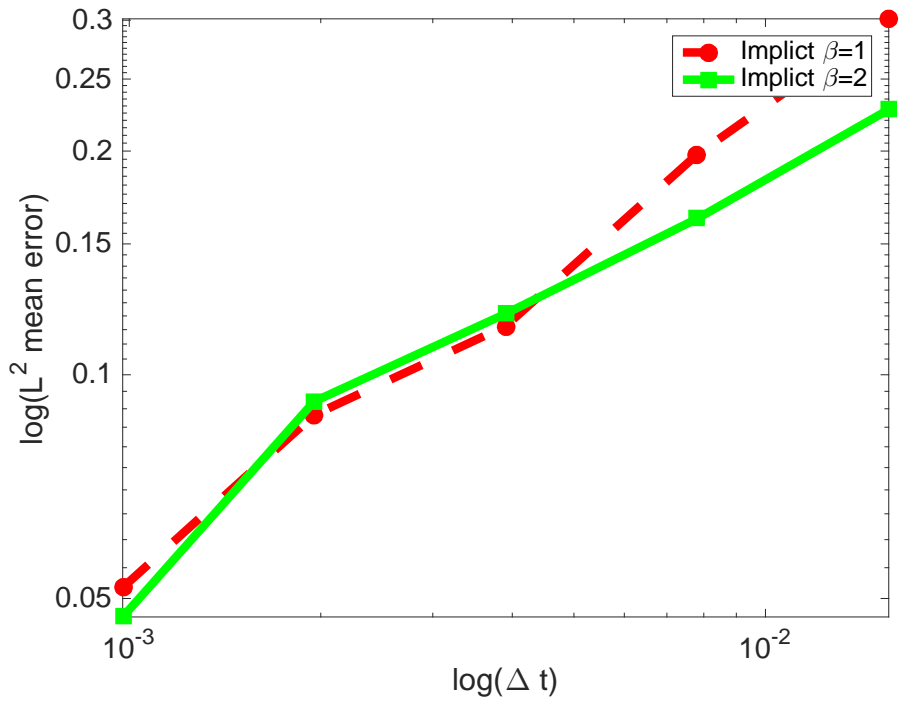}\\
 (b) 
 \includegraphics[width=0.7\textwidth]{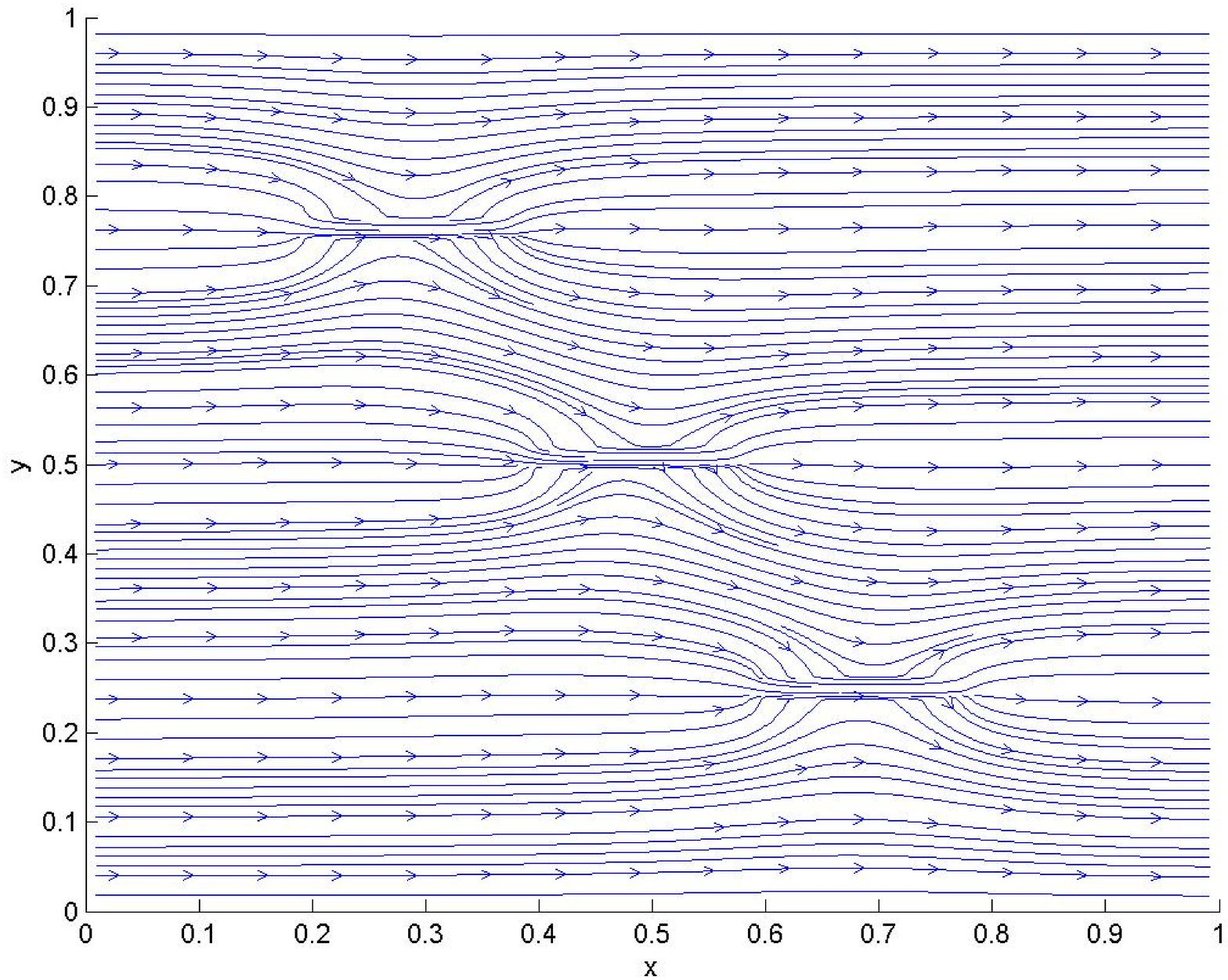}
  \end{center}
 \caption{(a) Convergence of the  implicit scheme for $\beta=1$,  and $\beta=2$ in \eqref{noise2} for SPDE \eqref{reactiondif1}.
 The orders of convergence in time  are $0.62$  for $\beta=1$, $0.54$ for $\beta=2$.  The total number of samples used is 100. The graph of the streamline $\mathbf{q}$ is given at  (b).}
 \label{FIGII}
 \end{figure}
 In \figref{FIGII}, we can observe the convergence of the the implicit scheme for two noise's parameters.
 Indeed, the order of convergence in time is $0.62$  for $\beta=1$ and $0.54$ for $\beta=2$. 
 These orders are close to the theoretical  orders $0.5$ obtained in \thmref{mainresult1} for $\beta=1$ and $\beta=2$.
%
%
\newpage


\end{document}